%% file: main.tex
\documentclass[]{article}
\usepackage{arxiv}
\input{header}

\author{%
	Dmitriy Metelev \\
	Moscow Institute of Physics and Technology \\
	\texttt{dimik701@gmail.com} \\
	\And
	Alexander Rogozin \\
	Moscow Institute of Physics and Technology \\
	\texttt{aleksandr.rogozin@phystech.edu} \\
	\AND
	Alexander Gasnikov \\
	Moscow Institute of Physics and Technology \\
	\texttt{gasnikov@yandex.ru} \\
	\And
	Dmitry Kovalev \\
	King Abdullah University of Science\\ and Technology \\
	\texttt{dmitry.kovalev@kaust.edu.sa} \\
}

\usepackage[colorinlistoftodos,bordercolor=blue,backgroundcolor=blue!20,linecolor=blue,textsize=scriptsize]{todonotes}
\newcommand{\ar}[1]{\textcolor{black}{#1}} % Alexander Rogozin's comments

\begin{document}
    \title{Decentralized Saddle-Point Problems with Different Constants of Strong Convexity and Strong Concavity}

    \author{
        Dmitriy Metelev\textsuperscript{1}\thanks{CONTACT Dmitriy Metelev. Email:    dimik701@gmail.com}, Alexander Rogozin\textsuperscript{1,2}, Alexander Gasnikov\textsuperscript{1,2,3}, Dmitry Kovalev\textsuperscript{4}\\
        \textsuperscript{1}Moscow Institute of Physics and Technology, Russia\\ \textsuperscript{2}HSE University, Russia\\
        \textsuperscript{3}ISP RAS Research Center for Trusted Artificial Intelligence, Russia\\
        \textsuperscript{4}King Abdullah University of Science and Technology, Saudi Arabia
    }
    
    \maketitle
    
    \begin{abstract}
        Large-scale saddle-point problems arise in such machine learning tasks as GANs and linear models with  affine constraints. In this paper, we study distributed saddle-point problems (SPP) with strongly-convex-strongly-concave smooth objectives that have different strong convexity and strong concavity parameters of composite terms, which correspond to $min$ and $max$ variables, and bilinear saddle-point part. \ar{We consider two types of first-order oracles: deterministic (returns gradient) and stochastic (returns unbiased stochastic gradient)}. Our method works in both cases and takes several consensus steps between oracle calls.
        %{\color{red}{For deterministic oracle the linear rate of convergence was also obtained in the case of not strongly convex-concave SPP.}}
        % The data samples are distributed between the nodes that are connected by a decentralized network.
    \end{abstract}


    %\section{Problem statement}
    \section{Introduction}
    
    In this paper we study saddle-point problems (SPP) \ar{with two composite terms and a bilinear part}
    \begin{equation}\label{problem_main}
        \min_{x \in \R^{d_x}}\max_{y \in \R^{d_y}}F(x, y) = f(x) + y^\top \mA x - g(y),
    \end{equation}
    where function $f$ is $\mu_x$-strongly convex ($\mu_x>0$) and $L_x$-smooth and function $g$ is $\mu_y$-strongly convex ($\mu_y>0$) and $L_y$-smooth. The interest to this class of problems has grown in the last few years \cite{zhang2019lower,alkousa2020accelerated,lin2020near,wang2020improved,tominin2021accelerated,kovalev2022first} due to the general growth of interest in SPP in ML community. \ar{The lower complexity bound obtained in \cite{zhang2019lower}:}
    \begin{equation}\label{rate}
        \mathcal{O}\left( \left( \sqrt{\frac{L_x}{\mu_x}} + \sqrt{\frac{\lambda_{max}(A^TA)}{\mu_x\mu_y}} + \sqrt{\frac{L_y}{\mu_y}} \right)\log{\frac1\varepsilon} \right).
    \end{equation}
    \ar{Optimal methods which work according to the lower bound were} obtained independently and almost at the same time in \cite{minimax2112,similar_determ2201,jin2022sharper}. For the case $\mu_x = \mu_y$, $L_x = L_y$ the lower bound and optimal methods were known much earlier \cite{nemirovski1983problem,tseng2000modified,lan2020first}. 
    
    \ar{The case of} stochastic oracle, in which we have access to unbiased stochastic gradients of $f$ and $g$, is studied much less in the literature. For example, in \cite{couplin_term2111} a non-bilinear SPP was considered and composite terms $f$, $g$ were assumed to be proximal-friendly. We study a general case when $f$ and $g$ may  be not proximal-friendly.
    
    The problem \eqref{problem_main} often arises in decentralized optimization \cite{boyd2011distributed,yarmoshik2022decentralized} and can be considered as a particular case (when $g \equiv 0$) of decentralized convex optimization problems with affine constraints \cite{decent_review2011}. It is obvious that in such applications we have $L_y \gg \mu_x = \mu_y \sim \varepsilon$, where $\varepsilon$ is a desired accuracy in duality gap.\footnote{Strictly speaking, we must put $\mu_y = 0$, since $g\equiv 0$. But without limiting generality, we can consider $\mu_y$ to be as small as $\varepsilon$ due to the regularization trick \cite{wang2020improved,rogozin2021decentralized}.} The state-of-the-art results for decentralized SPP \eqref{problem_main} \ar{proposed in papers \cite{beznosikov2020distributed,rogozin2021decentralized,kovalev2022optimal,LuoDecentrEXtra} (both for deterministic and stochastic oracles)} require $\mu_x = \mu_y$, $L_x = L_y$.

    The contributions of this paper are summarized as follows.
    \begin{enumerate} 
        \item (Sensitivity) we generalize \eqref{rate} in the case when oracle returns inexact gradients of $f$ and $g$ \cite{First-Order-Methods-of-Smooth-Convex-Optimization-with-Inexact-Oracle};
        \item (Stochasticity) by using sensitivity analysis and standard batch-technique (e.g. see \cite{gasnikov2022power}) we generalize \eqref{rate} in the case when oracle returns stochastic gradients of $f$ and $g$ (with different variances);
        \item (Decentralization) by using consensus-projection trick from \cite{decentr2009,stoch_decentr2103} and sensitivity analysis we generalize stochastic version of \eqref{problem_main} for decentralized setup.
    \end{enumerate}
    Note that we could try to implement the plan above starting with arbitrary optimal method from \cite{minimax2112,similar_determ2201,jin2022sharper} that has complexity bound \eqref{rate}. However, we are definitely preferred \cite{minimax2112}, because the results of \cite{minimax2112} include also the situation when $\mu_x = \mu_y =0$, but we still have a linear rate of convergence \cite{ibrahim2020linear,alkousa2020accelerated}. This gives us a stochastic generalization in non-convex-non-concave setup. But at the same time we can only get results for the strongly-convex-strongly-concave setup in decentralized case, it will be discussed in more detail below. 
    
    The \textit{drawbacks} of the proposed approach are as follows:
    \begin{itemize}
        \item (Lack of overparametrization) Based on the proposed batch-technique we do not know how to replace variance determined on whole space to the variance determined only in the solution, see \cite{gorbunov2021stochastic,beznosikov2022stochastic} for $\mu_x = \mu_y$, $L_x = L_y$ and non-distributed setup.
        % \arogozin{I think we need to discuss "Lack of parametrization" more properly. AG: Trully speaking, I don't know what to add here:(}
        \item (Extra logarithmic multiplier) Consensus-projection procedure leads to the \ar{addition} of an extra logarithmic multiplier (on a desired accuracy) in comparison with direct approaches, which was clearly demonstrated in the case $\mu_x = \mu_y$, $L_x = L_y$ in \cite{beznosikov2021near} and \cite{kovalev2022optimal}.
    \end{itemize}
    The \textit{advantages} of the proposed approach are as follows:
    \begin{itemize}
        \item (Universality) The idea to propose general scheme that allows to build optimal decentralized stochastic methods based on non-distributed deterministic ones seems to be quite attractive \cite{gasnikov2022power}. But for the moment this was done only for standard Nesterov's accelerated (momentum) method \cite{stoch_decentr2103}. The acceleration from \cite{minimax2112} is much more difficult\footnote{It seems that \cite{minimax2112} is the most tricky approach among \cite{minimax2112,similar_determ2201,jin2022sharper}, which use significantly new ideas of acceleration rather than standard ones \cite{nesterov2018lectures,lin2020accelerated}.} So the starting point of the plan (sensitivity analysis) required significant generalization of the results from \cite{First-Order-Methods-of-Smooth-Convex-Optimization-with-Inexact-Oracle}, which was used in \cite{stoch_decentr2103}. The results obtained in sensitivity part of this paper building a bridge to much wider class of optimal modern non-distributed non-stochastic methods. It is very important to note that proposed scheme preserve optimality (up to a logarithmic factor) of input method for output one (see \cite{rogozin2021decentralized} for the lower bound in deterministic setup).   
        \item (Average constants) As well as in the works \cite{decentr2009,stoch_decentr2103,beznosikov2021near} the complexity bounds (communication steps, oracle calls) determined by the average (among all the nodes) smoothness, strong convexity (concavity) and variance constants rather than the worse ones, which is typical for any other (which do not use dual oracle) approaches \cite{decent_review2011}.
        \item (Time-varying networks) As well as in the works \cite{decentr2009,stoch_decentr2103,beznosikov2021near} our results can be easily generalized on time-varying networks, rather than almost optimal loop-less approaches that are much more tricky and developed at the moment only for optimization problems \cite{kovalev2021adom,kovalev2021lower,li2021accelerated,song2021optimal}.
    \end{itemize}
    % \arogozin{1. Average constants: should we add letters to ease the reading? 2. What is meant by "direct" approaches? Primal approaches that do not use dual oracle? AG: I correct a little bit the text.}
    
    An alternative approach that could be used is based on decentralized Catalyst envelope \cite{tian2021acceleration}. In this approach we could build an optimal method for SPP \eqref{problem_main} based on decentralized accelerated method for optimization problem formulation \cite{kovalev2020optimal,kovalev2021adom,kovalev2021lower,li2021accelerated,song2021optimal}. We could try to build an optimal (up to a logarithmic factors) method in the same way as it was done in non-distributed setup \cite{lin2020near,tominin2021accelerated}. Unfortunately, this approach: a) can deal only with deterministic oracles;  b) is characterized by the worst-case constants (not the average ones); c) leads to the third degree (at least) of the logarithmic factor \cite{lin2020near,tominin2021accelerated} that is worse than in described above approach.\footnote{May be this drawback could be eliminated over time, like it was done in \cite{kovalev2022first} for non-distributed setup.} 
    
    Prior to our paper, we were not aware of papers that considered decentralized strongly-convex-strongly-concave SPP with different constants of strong convexity and strong concavity, even in deterministic case. The most competitive paper to ours is \cite{kovalev2022optimal} which considers equal constants of strong convexity and strong concavity. Therefore, the result of \cite{kovalev2022optimal} would contain $1/\min\{\mu_x,\mu_y\}$, whereas our result is $1/\sqrt{\mu_x\mu_y}$, which may be much better.

    \subsection{Basic definitions and assumptions}
    \noindent
    
    %Initially, we're going to introduce some basic notions
    
    %Problem \ref{problem_main} is going to be solved in a decentralized way by the network of agents. We want to describe it more formally.
    
    The paper is mainly focused on the decentralized saddle point problems. Namely, we aim at solving
    \begin{align}\label{eq:spp_sum_type}
        \min_{x\in\R^{d_x}} \max_{y\in\R^{d_y}} \frac{1}{n} \sum_{i=1}^n f_i(x) + y^\top\mA x - g_i(y).
    \end{align}
    We assume local functions $f_i$ and $g_i$ to be strongly convex and smooth.
    \ar{
    \begin{definition}
        Function $h(x):~ \R^d\to\R$ is called $L$-smooth if for any $x, y\in\R^d$ it holds
        \begin{align*}
            \norm{\nabla h(y) - \nabla h(x)}\leq L\norm{y - x}.
        \end{align*}
    \end{definition}
    \begin{definition}
        Function $h(x):~ \R^d\to\R$ is called $\mu$-strongly convex if for any $x, y\in\R^d$ it holds
        \begin{align*}
            h(y)\geq h(x) + \angles{\nabla h(x), y - x} + \frac{\mu}{2}\norm{y - x}^2.
        \end{align*}
    \end{definition}
    }
    \begin{assumption}
        Function $f_i(x):\R^{d_x} \rightarrow \R$ is $\mu_{x,i}$-strongly convex and $L_{x,i}$-smooth function, $\mu_{x,i} \geq 0$, $L_{x,i} > 0$. There exists $i$ such that $\mu_{x,i}>0$.
    \end{assumption}
    
    \begin{assumption}
        Function $g_i(y):\R^{d_y} \rightarrow \R$ is $\mu_{y,i}$-strongly convex and $L_{y,i}$-smooth function, $\mu_{y,i} \geq 0$, $L_{y,i} > 0$. There exists $i$ such that $\mu_{y,i}>0$.
    \end{assumption}
    
    Every agent's oracle has access only to stochastic gradients of $f_i(x)$ and $g_i(y)$, we denote them as $\nabla f_i(x, \xi_x)$ and $\nabla g_i(y, \xi_y)$ correspondingly, where $\xi_x$ and $\xi_y$ are random variables.
    
    \begin{assumption}\label{ass:3}
        For each i there exists $\sigma_{f, i}^2$ such that
        \begin{equation*}
            \E_{\xi_x}\nabla f_i(x, \xi_x) = \nabla f_i(x),~
            \E_{\xi_x}\sqn{\nabla f_i(x, \xi_x)-\nabla f_i(x)} \leq \sigma_{f, i}^2.
        \end{equation*}
    \end{assumption}
    
    \begin{assumption}\label{ass:4}
        For each i there exists $\sigma_{g, i}^2$ such that
        \begin{equation*}
            \E_{\xi_y}\nabla g_i(y, \xi_y) = \nabla g_i(y),~
            \E_{\xi_y}\sqn{\nabla g_i(y, \xi_y)-\nabla g_i(y)} \leq \sigma_{g, i}^2.
        \end{equation*}
    \end{assumption}
    
    \begin{assumption}\label{ass:5}
        %Coupling matrix $\mA \in \R^{d_y \times d_x}$.
        There exist constants $L_{xy} > 0$, $\mu_{xy}, \mu_{yx} \geq 0$ such that
        \begin{align*}
            L_{xy}^2 &\geq \lambda_{max}\left(\mA^T\mA \right)=\lambda_{max}\left( \mA \mA^T \right)\label{as:1}, \\
            \mu_{xy}^2 &\leq
            \begin{cases}
                \lambda_{min}^+(\mA\mA^\top), &\text{if }~\nabla g(y, \xi) \in \range\mA \text{ for all $\xi$ and y $\in \R^{d_y}$} \\ 
                \lambda_{min}(\mA\mA^\top), &\text{otherwise}
            \end{cases} \\
            \mu_{yx}^2 &\leq
            \begin{cases}
                \lambda_{min}^+(\mA^\top\mA), &\text{if~ $\nabla f(x, \xi) \in \range\mA$ for all $\xi$ and x $\in \R^{d_x}$} \\ 
                \lambda_{min}(\mA^\top\mA), &\text{otherwise}
            \end{cases}
        \end{align*}
        
        where $\lambda_{min}(\cdot)$,  $\lambda_{min}^+(\cdot)$ and  $\lambda_{max}(\cdot)$ denote the smallest, smallest positive and largest eigenvalue of a matrix, respectively, and $\range(\cdot)$ denotes the range space of a matrix.
    \end{assumption}
    
    % 	We also assume that our functions $f$ and $g$ have the following form.
    % 	\begin{equation}\label{ass:f}
        % 		f(x) = \frac{1}{n}\sum_{i=1}^n f_i(x),
        % 	\end{equation}
    % 	\begin{equation}\label{ass:g}
        % 		g(y) = \frac{1}{n}\sum_{i=1}^n g_i(y).
        % 	\end{equation}
    
    % 	According to \eqref{ass:f} and \eqref{ass:g}, problem \ref{problem_main} can be rewritten as follows:
    % 	\begin{equation}\label{problem_reformulated}
        % 		\min_{x \in \R^{d_x}}\max_{y \in \R^{d_y}}f(x, y) = f(x) + y^\top \mA x - g(y).
        % 	\end{equation}
    
    Each node holds its own copy of global variables $x$ and $y$, and we introduce matrices $X=[x_1 x_2 \dots x_n]\in \R^{d_x \times n}$ and $Y=[y_1 y_2 \dots y_n]\in \R^{d_y\times n}$. We also denote $\ovl{x} = \frac{1}{n}\sum_{i=1}^n x_i$, $\ovl{y} = \frac{1}{n}\sum_{i=1}^n y_i$ and $\ovl{X} = (\ovl{x} \dots \ovl{x})\in \R^{d_x \times n}$, $\ovl{Y} = (\ovl{y} \dots \ovl{y})\in \R^{d_y \times n}$.
    %Decentralized algorithm operates by global values, thus we denote global functions of $X$ and $Y$.
    
    Introduce functions
    %$F(X):\R^{d_x \times n} \rightarrow \R$ and $G(X):\R^{d_y \times n} \rightarrow \R$ such that
    \begin{equation}
        F(X) = \sum_{i=1}^n f_i(x_i),~
        G(Y) = \sum_{i=1}^n g_i(y_i).
    \end{equation}
    % 	The corresponding gradients write as
    % 	\begin{align*}
        % 	    \nabla F(X) &= (\nabla f_1(x_1)~ \nabla f_2(x_2)~ \dots~ \nabla f_n(x_n))\in \R^{d_x \times n}, \\
        % 	    \nabla G(Y) &= (\nabla g_1(y_1)~ \nabla g_2(y_2)~ \dots~ \nabla g_n(y_n))\in \R^{d_y \times n}.
        % 	\end{align*}
    
    \begin{definition}
        Let $\mathcal{S}$ be a nonempty set of solutions of saddle-point problem. Then we call a pair of vectors $(x, y)$ an $\epsilon$-solution to SPP \eqref{eq:spp_sum_type} for given accuracy $\epsilon > 0$ if it satisfies
        \begin{equation}
            \min_{(x^*, y^*) \in \mathcal{S}} \max \{\sqn{x-x^*}, \sqn{y-y^*}\} \leq \epsilon.
        \end{equation}
    \end{definition}
    
    The complexity of the algorithm is comprised of the number of communications between nodes and the number of computations of matrix-vector products $\mA x$,  $y^\top \mA$ and stochastic first-order oracle calls $\nabla f(x, \xi)$, $\nabla g(y, \xi)$.
    %	Our goal is to propose an algorithm to find $\epsilon$-solution to Problem~\eqref{eq:spp_sum_type} with least complexity.
    
    \begin{remark}
        Our analysis provides an algorithm whose complexity linearly depends on global parameters 
        \begin{align*}
            L_{x} &= \frac{1}{n}\sum_{i=1}^n L_{x, i},~ \mu_{x} = \frac{1}{n}\sum_{i=1}^n \mu_{x, i},~ \sigma_f^2 = \frac{1}{n}\sum_{i=1}^n \sigma_{f, i}^2, \\
            L_{y} &= \frac{1}{n}\sum_{i=1}^n L_{y, i},~ \mu_{y} = \frac{1}{n}\sum_{i=1}^n \mu_{y, i},~ \sigma_g^2 = \frac{1}{n}\sum_{i=1}^n \sigma_{g, i}^2
            % 	L_{y}, \mu_{y}, \sigma_g^2 
        \end{align*}
        Local parameters are defined as $L_{lx} = \max_i\{L_{x, i}\}$, $\mu_{lx} = \min_i\{\mu_{x, i}\}$, $L_{ly} = \max_i\{L_{y, i}\}$, $\mu_{ly} = \min_i\{\mu_{y, i}\}$.
    \end{remark}

    \subsection{Main idea: approximation of non-distributed algorithm via consensus subroutine}
    \noindent
    
    Problem~\eqref{eq:spp_sum_type} can be written as
    \begin{equation}\label{problem:2}
        % F(X, Y)
        \min_{X \in \mathcal{C}_{d_x}}\max_{Y \in \mathcal{C}_{d_y}} \frac{1}{n}\sum_{i=1}^n f_i(x_i) + y_i^\top \mA x_i - g_i(y_i),
    \end{equation}
    where 
    \begin{center}
    $\mathcal{C}_{d_x} = \braces{X\in\R^{d_x \times n}:~ x_1 = \ldots = x_n}$ and $\cC_{d_y} = \braces{Y\in\R^{d_y \times n}:~ y_1 = \ldots = y_n}$.
    \end{center}
    
    We want to expand the existing algorithm to the decentralized case of finite sum SPP, and it would be very convenient if we could pass the information about gradients from every node to a single center, average the gradients and pass them back. It would be the same if we made one gradient iteration and averaged variables at each node. Therefore, the main idea of consensus is to make this "averaging" as precise as possible. The only thing we can do is to transmit the information to neighbors. There is a number of solutions in the literature of how to make this "averaging", which often translate to the properties, provided in the "Consensus" section. The relevance of such properties is discussed in more detail in "Consensus" section in \cite{decentr2009}.
    
    So we want to make our "averaging" as close as possible to the ideal one, but there are always "inexactness" in equality of node values, which is expressed by the distance from $Z$ to $C_{d_z}$ subspace. Hence, we need the Lemma \ref{l1} which helps us to transform the "averaging inexactness" to $\delta$ constant in the definition of inexact oracle which introduced in the section below. Therefore, adjusting the number of consensus iterations, we bound the $\delta$ constant.
    
    We use a fixed number of consensus iterations, so to get the theoretical guarantees that $\delta$ constant would be sufficiently small during the algorithm. We need to prove that the possible "inexactness" after a gradient iteration is bounded and that this bound can be expressed polynomially through the initial constants.
    Since we work in a stochastic setup, we can only get the bound for expectation of "inexactness" after gradient iteration, therefore get the bound for expectation of "inexactness" after consensus, which guarantees us the bound in expectation of $\delta$ constant.
    
    To sum up, implementation of this plan requires sensitive and stochastic (in stochastic decentralized setup) analysis of the algorithm used in decentralized case.
    
    In our case in particular, we take Algorithm 1 from \cite{minimax2112} as a basis for our decentralized algorithm. For this purpose, we need to expand results from \cite{minimax2112} on the case of stochastic inexact oracle.
    
    \section{Inexact oracle framework}
    \subsection{Preliminaries}
    We will use the definition of $(\delta, L, \mu)$-oracle. Let $h(x)$ be a convex function defined on a convex set $Q\subseteq \mathbb{R}^m$. We say that $(h_{\delta, \mu, L}(x), s_{\delta, \mu, L}(x))$ is a $(\delta, L, \mu)$-model of $h(x)$ at  point $x\in Q$ if for all $y \in Q$ holds
    
    $$\frac{\mu}{2}\|y - x\|^2 \leq h(y) - \left( h_{\delta, \mu, L}(x) + \langle s_{\delta, \mu, L}(x), y - x \rangle \right) \leq \frac{L}{2}\|y - x\|^2 + \delta.$$
    
    With slight abuse of notation, we say \ar{that} $\nabla f_{\delta}(x)$ is $(\delta, L, \mu)$-model of $f(x)$ at point $x$ if there exists $c$ such that $(c, \nabla f_{\delta}(x))$ is a $(\delta, L, \mu)$-model of $f(x)$ at  point $x$. \ar{Constants} $L$ and $\mu$ are \ar{derived} from the context.
    
    \subsection{Inexact oracle for $f(x)$}
    
    Consider the sequence $\{f_i(x)\}_{i=1}^n$, $f_i$ is $L_i$-smooth and $\mu_i$ strongly convex.
    
    Define $f(\ovl{x})=\frac1n F(\ovl{X})$.
    
    \begin{lemma}\label{l1}
        Let $X\in \mathbb{R}^{d \times n}$ and $\|X-\overline{X}\|^2\leq\delta'$
        
        Define
        $$\delta = \frac{1}{2n}\left( \frac{L_l^2}{L_g} + \frac{2L_l^2}{\mu_g} + L_l - \mu_l \right)\delta',$$
        $$f_{\delta,L,\mu}(\overline{x}, X) = \frac{1}{n}\left( F(X) + \langle \nabla F(X), \overline{X} - X \rangle + \frac{1}{2} \left( \mu_l - \frac{2L_l^2}{\mu_g} \right)\|\overline{X} - X\|^2 \right),$$
        $$g_{\delta, L, \mu}(\overline{x}, X) = \frac{1}{n}\sum_{i=1}^n \nabla f_i(x_i).$$
        
        Then $\left( f_{\delta,L,\mu}(\overline{x}, X), g_{\delta, L, \mu}(\overline{x}, X)\right)$ is a $\left( \delta, 2L_g, \mu_g/2 \right)$-model of $f$ at point $\overline{x}$.
        
    \end{lemma}
    
    \begin{proof}
        The proof is provided in Lemma 2.1 in \cite{decentr2009}.
    \end{proof}
    
    This theorem assumes that $L_g, \mu_g$ variables reflect global constants, while $L_l, \mu_l$ reflect local constants.
    
    This is a fundamental theorem for our analysis because it allows us to use inexact oracle, through which we can use such an approximations (consider the case of $\|X-\overline{X}\|^2\leq\delta'$).
    $$\frac{\mu_g}{4}\norm{\ovl{y}-\ovl{x}}^2 \leq f(\ovl{y}) - f_{\delta,L,\mu}(\overline{x}, X) - \langle g_{\delta, L, \mu}(\overline{x}, X),\ovl{y}-\ovl{x}\rangle \leq L_g\norm{\ovl{y}-\ovl{x}}^2+\delta$$
    
    If we could calculate the gradient at the point $\ovl{x}$, we could get rid of $\delta$, but this is the price for decentralization, so we are trying to call the oracle at the point that is as close as possible to $C_z$.  The iteration of our decentralized algorithm is performed as follows: being in the $\sqrt{\delta'}$ neighborhood of $C_z$, we iterate the basic algorithm, then make projection to $\sqrt{\delta'}$ neighborhood of $C_z$.
    
    \subsection{Inexact APDG}
    In this section we consider the Algorithm 1 from \cite{minimax2112}, but for the sake of using it as a basis for decentralized \ar{SPP} with consensus subroutine we need its inexact variant.
    
    \begin{algorithm}[H]\caption{Inexact APDG}\label{alg:APDG}
        \begin{algorithmic}
            \State \textbf{Input:} {$\eta_x, \eta_y, \alpha_x, \alpha_y, \beta_x, \beta_y > 0$, $\tau_x, \tau_y, \sigma_x, \sigma_y\in \left(0, 1\right]$, $\theta\in \left(0, 1\right)$
                
                \State $x^0_f = x^0 \in \range \mA^\top$
                
                \State $y^0_f = \overline{y}^{-1} = y^0 \in \range \mA$}
            
            \For{$k=0,1,2,\dots$}
            \State $y_m^k = y^k + \theta\left(y^k-y^{k-1}\right)$
            
            \State $x_g^k = \tau_x x^k + \left(1 - \tau_x\right)x_f^k$
            
            \State $y_g^k = \tau_y y^k + \left(1 - \tau_y\right)y_f^k$
            
            \State $x^{k+1}=x^k+\eta_x\alpha_x\left(x_g^k - x^k\right) - \eta_x\beta_x A^T\left(A x^k-\nabla g_{\delta}(y_g^k)\right) - \eta_x\left(\nabla f_{\delta}(x_g^k) + A^T y_m^k\right)$
            
            \State $y^{k+1} = y^k + \eta_y\alpha_y\left(y_g^k - y^k\right) - \eta_y\beta_y A\left(A^T y^k+\nabla f_{\delta}(x_g^k)\right) - \eta_y\left(\nabla g_{\delta}(y_g^k) - A x^{k+1}\right)$
            
            \State $x_f^{k+1} = x_g^k + \sigma_x\left(x^{k+1} - x^k\right)$
            
            \State $y_f^{k+1} = y_g^k + \sigma_y\left(y^{k+1} - y^k\right)$
            
            \EndFor
        \end{algorithmic}
    \end{algorithm}
    
    \begin{theorem}\label{th2}
        Let $f(x)$ be $\mu_x$-strongly convex and $L_x$ smooth, and let $g(y)$ be $\mu_y$-strongly convex and $L_y$ smooth (where $\mu_x,\mu_y > 0$). Let us have access to $(\delta_x, L_x, \mu_x)$-oracle of $f$ and $(\delta_y, L_y, \mu_y)$-oracle of $g$. Also suppose Assumption \ref{ass:5} holds. Then there exist parameters of Algorithm \ref{alg:APDG} such that
        $$\|x^k - x^*\|^2 \leq \frac{1}{3L_{xy}}\sqrt{\frac{\mu_y}{\mu_x}}\left( \theta^k\Psi^0 + \frac{4}{(1-\theta)^2}(\delta_x+\delta_y) \right),$$
        $$\|y^k - y^*\|^2 \leq \frac{1}{4L_{xy}}\sqrt{\frac{\mu_x}{\mu_y}}\left( \theta^k\Psi^0 + \frac{4}{(1-\theta)^2}(\delta_x+\delta_y) \right),$$
        $$\frac{1}{1-\theta} \leq 4+4\max\left\{\sqrt{\frac{L_x}{\mu_x}}, \sqrt{\frac{L_y}{\mu_y}}, \frac{L_{xy}}{\sqrt{\mu_x\mu_y}}\right\}.$$
        Here $\Psi^0$ depends polynomially on initial values.
    \end{theorem}
    Theorem \ref{th2} is a consequence of Theorem~\ref{th4} with $\sigma_f^2=\sigma_g^2=0$. \ar{We discuss Theorem~\ref{th4} in Section~\ref{sec:stochastic_inexact_case} and its proof is presented in Appendix~\ref{app:inexact_setting}.}
    
    \section{Stochastic inexact case}\label{sec:stochastic_inexact_case}
    \subsection{Inexact stochastic oracle}
    Consider the sequence $\{f_i(x)\}_{i=1}^n$, $f_i$ is $L_i$-smooth and $\mu_i$ strongly convex. 
    \begin{lemma}\label{lemma:3}
        
        Let $X\in \mathbb{R}^{d \times n}$ and $\|X-\overline{X}\|^2\leq\delta'$.
        
        Define
        $$\delta = \frac{1}{2n}\left( \frac{L_l^2}{L_g} + \frac{2L_l^2}{\mu_g} + L_l - \mu_l \right)\delta',$$
        $$f_{\delta,L,\mu}(\overline{x}, X) = \frac{1}{n}\left( F(X) + \langle \nabla F(X), \overline{X} - X \rangle + \frac{1}{2} \left( \mu_l - \frac{2L_l^2}{\mu_g} \right)\|\overline{X} - X\|^2 \right),$$
        $$g_{\delta, L, \mu}(\overline{x}, X) = \frac{1}{n}\sum_{i=1}^n \nabla f_i(x_i),$$
        $$\widetilde{g}_{\delta, L, \mu}(\overline{x}, X) = \frac{1}{n}\sum_{i=1}^n\frac{1}{r}\sum_{j=1}^r \nabla f_i(x_i, \xi_i^j).$$
        
        Then $\left( f_{\delta,L,\mu}(\overline{x}, X), g_{\delta, L, \mu}(\overline{x}, X)\right)$ is a $\left( \delta, 2L_g, \mu_g/2 \right)$-model of $f$ at point $\overline{x}$. Moreover, we have
        \begin{align*}
            &\mathbb{E}\widetilde{g}_{\delta, L, \mu}(x) = g_{\delta, L, \mu}(x), \\
            &\mathbb{E}\|\widetilde{g}_{\delta, L, \mu}(x)-g_{\delta, L, \mu}(x)\|^2 \leq \frac{\sum_{i=1}^n\sigma_{f, i }^2}{n^2r}=\frac{\sigma_f^2}{nr}.
        \end{align*}
    \end{lemma}
    \begin{proof}
        See Lemma 4.1 from \cite{stoch_decentr2103}.
    \end{proof}
    
    This lemma is similar to the Lemma \ref{l1} and follows the same idea (to transform the "inexactness" in equality of node values in $\delta$ constant in inexact oracle), but also considers stochastic case.
    
    \subsection{Inexact stochastic APDG}

    \begin{definition}
        There are a sequence of independent random variables for stochastic oracle, which represent a history of stochastic process. Let $\xi^k = (\xi_x^1, \xi_x^2, \dots, \xi_x^k, \xi_y^1, \xi_y^2, \dots, \xi_y^k)$. $\xi^k_x=\{\xi_x^{k,i}\}_{i=1}^{r_f}$ and $\xi^k_y=\{\xi_y^{k,j}\}_{j=1}^{r_g}$ are random variables for batches at iteration k, $r_f,r_g$ denote the batch sizes for $f$ and $g$ respectively. 
    \end{definition}

    \begin{definition}
        For every iteration $k$ of the algorithm, the oracle has access only to stochastic inexact gradients $\nabla f_{\delta}(x, \xi_x^{k,i}, \xi^{k-1})$ and $\nabla g_{\delta}(y, \xi_y^{k,j},  \xi^{k-1})$, where $1 \leq i \leq r_f$, $1 \leq j \leq r_g$. Let $\nabla f_{\delta}^i(x,  \xi^{k-1})=\E_{\xi_x^{k,i}}\nabla f_{\delta}(x, \xi_x^{k,i},  \xi^{k-1})$ and $\nabla g_{\delta}^j(y,  \xi^{k-1})=\E_{\xi_y^{k,j}}\nabla g_{\delta}(y, \xi_y^{k,j},  \xi^{k-1})$. $\nabla^r f_{\delta}(x, \xi_x^k, \xi^{k-1})$ and $\nabla^r g_{\delta}(y, \xi_y^k, \xi^{k-1})$ are batched gradients.
    \end{definition}

    \begin{assumption}\label{ass:6}(Inexact oracle property)\\
        \ar{Let $\nabla g_{\delta}^j(y, \xi^{k-1})$ denote the $(\delta_y(\xi^{k-1}), L_y, \mu_y)$-model of $g$ at point $y$ and $\nabla f_{\delta}^i(x, \xi^{k-1})$ denote the $(\delta_x(\xi^{k-1}), L_x, \mu_x)$-model of $f$ at point $x$, respectively. We assume that}
        $\D_{\xi_x^{k,i}}\nabla f_{\delta}(x, \xi_x^{k,i},  \xi^{k-1}) \leq \sigma_f^2$ and $\D_{\xi_y^{k,j}}\nabla g_{\delta}(y, \xi_y^{k,j}, \xi^{k-1}) \leq \sigma_g^2$. Moreover, we assume that $\E_{\xi^{k-1}} \delta_x(\xi^{k-1}) \leq \delta_x$ and $\E_{\xi^{k-1}} \delta_y(\xi^{k-1}) \leq \delta_y$.
    \end{assumption}

    \noindent\ar{For brevity we write $\nabla g_k=\nabla^r g_{\delta}(y_g^k, \xi_y^k, \xi^{k-1})$, $\nabla f_k = \nabla^r f_{\delta}(x_g^k, \xi_x^k, \xi^{k-1}).$}
    \begin{algorithm}[H]\caption{Inexact stochastic APDG}\label{alg:ISAPDG}
        \begin{algorithmic}
            \State \textbf{Input:} {$\eta_x, \eta_y, \alpha_x, \alpha_y,  \beta_x, \beta_y > 0$, $\tau_x, \tau_y, \sigma_x, \sigma_y\in \left(0, 1\right]$, $\theta\in \left(0, 1\right)$
                
                \State $x^0_f = x^0 \in \range \mA^\top$
                
                \State $y^0_f = \overline{y}^{-1} = y^0 \in \range \mA$}
            
            \For{$k=0,1,2,\dots$}
            \State $y_m^k = y^k + \theta\left(y^k-y^{k-1}\right)$\label{ISAPDG:line:1}
            
            \State $x_g^k = \tau_x x^k + \left(1 - \tau_x\right)x_f^k$\label{ISAPDG:line:2}
            
            \State $y_g^k = \tau_y y^k + \left(1 - \tau_y\right)y_f^k$\label{ISAPDG:line:3}
            
            \State $x^{k+1}=x^k+\eta_x\alpha_x\left(x_g^k - x^k\right) - \eta_x\beta_x A^T\left(A x^k-\nabla g_k\right) - \eta_x\left(\nabla f_k + A^T y_m^k\right)$\label{ISAPDG:line:4}
            
            \State $y^{k+1} = y^k + \eta_y\alpha_y\left(y_g^k - y^k\right) - \eta_y\beta_y A\left(A^T y^k+\nabla f_k\right) - \eta_y\left(\nabla g_k - A x^{k+1}\right)$\label{ISAPDG:line:5}
            
            \State $x_f^{k+1} = x_g^k + \sigma_x\left(x^{k+1} - x^k\right)$\label{ISAPDG:line:6}
            
            \State $y_f^{k+1} = y_g^k + \sigma_y\left(y^{k+1} - y^k\right)$\label{ISAPDG:line:7}
            
            \EndFor
        \end{algorithmic}
    \end{algorithm}
    
    Such general conditions are needed because they cover the properties of average values obtained after averaging Algorithm~\ref{alg:DAPDG}. As already said, we keep the expectations of values $\sqn{X-\ovl{X}}$ and $\sqn{Y-\ovl{Y}}$ small enough. By Lemma~\ref{lemma:3}, this implies that the expectations of gradients $\nabla g_{\delta}(y_g^k)=\E g_{\delta}^r(y_g^k, \xi_y^k, \xi^{k-1})$ and $\nabla f_{\delta}(x_g^k)=\E g_{\delta}^r(x_g^k, \xi_x^k, \xi^{k-1})$ are $(\delta_y, L_y, \mu_y)$-model of $g$ at point $y_g^k$ and $(\delta_x, L_x, \mu_x)$-model of $f$ at point $x_g^k$ correspondingly.
    
    \begin{theorem}\label{th4}
        Let $f(x)$ be $\mu_x$-strongly convex and $L_x$ smooth, $g(y)$ be $\mu_y$-strongly convex and $L_y$ smooth ($\mu_x,\mu_y\geq 0$). We have access to $(\delta_x^k, L_x, \mu_x)$-stochastic oracle of $f$ with variance upper bounded by $\sigma_f^2$ and $(\delta_y^k, L_y, \mu_y)$-stochastic oracle of $g$ with variance upper bounded by $\sigma_g^2$ at iteration $k$. We also denote batch size with $r$. Also suppose Assumption~\ref{ass:5} and Assumption~\ref{ass:6} hold in environment of Algorithm~\ref{alg:ISAPDG}. Then there exist different sets of parameters of Algorithm~\ref{alg:ISAPDG} such that
        $$\mathbb{E}\|x^k - x^*\|^2 \leq \frac{\omega}{3L_{xy}}\left( \theta^k \Psi^0 +  \frac{4}{(1-\theta)^2}(\delta_x+\delta_y)+\frac{\Sigma^2}{2(1-\theta)} \right),$$
        $$\mathbb{E}\|y^k - y^*\|^2 \leq \frac{1}{4L_{xy}\omega}\left( \theta^k \Psi^0 +  \frac{4}{(1-\theta)^2}(\delta_x+\delta_y)+\frac{\Sigma^2}{2(1-\theta)} \right),$$
        $$\Sigma^2 = \left( \frac{1}{L_x}+\frac{\omega}{L_{xy}} \right)\frac{\sigma_f^2}{r_f}+\left( \frac{1}{L_y}+\frac{1}{L_{xy}\omega} \right)\frac{\sigma_g^2}{r_g}.$$
        
        Here is the list of possible estimations depending on different constants:

        \begin{subequations}
            \begin{align}%\label{ests}
            \frac{1}{1-\theta} &\leq
            4 + 4\max\left\{\sqrt{\frac{L_x}{\mu_x}}, \sqrt{\frac{L_y}{\mu_y}},\frac{L_{xy}}{\sqrt{\mu_x\mu_y}}\right\}, \omega = \sqrt{\frac{\mu_y}{\mu_x}}, \label{eq:th4_bound_1}\\
            \frac{1}{1-\theta} & \leq
            4+8\max\left\{
            \frac{\sqrt{L_xL_y}}{\mu_{xy}},
            \frac{L_{xy}}{\mu_{xy}}\sqrt{\frac{L_x}{\mu_x}},
            \frac{L_{xy}^2}{\mu_{xy}^2}
            \right\}, \omega = \sqrt{\frac{\mu_{xy}^2}{2\mu_xL_x}}, \label{eq:th4_bound_2}\\
            \frac{1}{1-\theta} & \leq
            4+8\max\left\{
            \frac{\sqrt{L_xL_y}}{\mu_{yx}},
            \frac{L_{xy}}{\mu_{yx}}\sqrt{\frac{L_y}{\mu_y}},
            \frac{L_{xy}^2}{\mu_{yx}^2}
            \right\}, \omega = \sqrt{\frac{2\mu_yL_y}{\mu_{yx}^2}}, \label{eq:th4_bound_3}\\
            \frac{1}{1-\theta} &\leq 2+8\max\left\{
            \frac{\sqrt{L_xL_y}L_{xy}}{\mu_{xy}\mu_{yx}},
            \frac{L_{xy}^2}{\mu_{yx}^2},
            \frac{L_{xy}^2}{\mu_{xy}^2}
            \right\}, \omega = \frac{\mu_{xy}}{\mu_{yx}}\sqrt{\frac{L_y}{L_x}}. \label{eq:th4_bound_4}
        \end{align}
    \end{subequations}
    Here $\Psi^0$ depends polynomially on initial values.
    \end{theorem}
    The proof is provided in appendix \ref{app:inexact_setting}, including the constants values for every estimation.
    
    Each line in the list of possible estimations symbolises particular case in a convex-concave conditions. Firstly, \eqref{eq:th4_bound_1} expresses strongly-convex-strongly-concave case. Secondly, \eqref{eq:th4_bound_2} describes strongly-convex-concave and positive $\mu_{xy}$ case. Thirdly, line \eqref{eq:th4_bound_3} expresses convex-strongly-concave and positive $\mu_{yx}$ case. Forthly, \eqref{eq:th4_bound_4} covers convex-concave and positive $\mu_{xy}$, $\mu_{yx}$ case. Every case provides a linear convergence rate for corresponding class of problems. The last case is associated with convex-concave problem with square full-range matrix(in this case $\mu_{xy},\mu_{yx}>0$ and Assumption~\ref{ass:5} holds).
    
    Speaking about deterministic side, according to \cite{zhang2019lower} this upper bound is optimal in strongly-convex-strongly-concave case.  
    
    \section{\ar{Decentralized} algorithm and results}
    We take Algorithm 1 from \cite{minimax2112} as a basis for our method. This algorithm works according to lower bounds and also cover convex-concave case, therefore we obtain desirable non-decentralized stochastic generalization, which may be futher generalized to decentralized case.
    % \ar{
    % We adapt Algorithm~\ref{alg:ISAPDG} to be run in a decentralized setup. We assume that each node locally holds $f_i(x_i, y_i)$ and has access to a first-order stochastic oracle corresponding to $f_i$. For each node $i = 1, \ldots, m$ introduce batch sizes $r_{f,i}$ and $r_{g,i}$. Also for each node let $\xi_{x,i}^k = (\xi_{x,i}^1, \xi_{x,i}^2, \ldots, \xi_{x,i}^k)$
    % }
    
    The algorithm can be executed in a decentralized way due to "decentralized" property of mixing matrices in Consensus algorithm ("Consensus" section). Every node has its own sequence of independent random variables (different sequences are also independent), which is used in calls of stochastic oracle. \ar{Let $r_{f,i}$ and $r_{g,i}$ denote batch sizes at node $i$ for $f$ and $g$, respectively}, and let $\xi_{x, i}^k$ and $\xi_{y, i}^k$ be the sets of random variables on $i$'th node at iteration $k$. Also introduce $\xi_x^k = (\xi_{x, 1}^k \xi_{x, 2}^k \dots \xi_{x, n}^k)$, $\nabla^r F(X, \xi_x^k) = (\nabla^r f_1(x_1, \xi_{x, 1}^k) \nabla^r f_2(x_2, \xi_{x, 2}^k) \dots \nabla^r f_n(x_n, \xi_{x, n}^k))$. Here $\nabla^r f_i(x_i, \xi_{x, i}^k)$ is a batched gradient. Notations $\xi_y^k$ and  $\nabla^r G(X, \xi_y^k)$ are defined in the same way.
    
    \ar{In the algorithm below, we write $\nabla F_k = \nabla^r F(X_g^k, \xi_x^k)$, $\nabla G_k = \nabla^r G(Y_g^k, \xi_y^k)$ for brevity}.

    \begin{algorithm}[H]\caption{Decentralized APDG with consensus subroutine}\label{alg:DAPDG}
        \begin{algorithmic}
            \State \textbf{Input:} {$\eta_x, \eta_y, \alpha_x, \alpha_y, \beta_x, \beta_y > 0$, $\tau_x, \tau_y, \sigma_x, \sigma_y\in \left(0, 1\right]$, $\theta\in \left(0, 1\right)$
                
                \State $X^0_f = X^0 = \overline{X^0}$, $\overline{x}^0 \in \range \mA^\top$

                \State $Y^0_f = Y^{-1} = Y^0 = \overline{Y^0}$, $\overline{y}^0 \in \range \mA$}
            
            \For{$k=0,1,2,\dots$}
            \State $Y_m^k = Y^k + \theta\left(Y^k-Y^{k-1}\right)$\label{DAPDG:line:1}
            
            \State $X_g^k = \tau_x X^k + \left(1 - \tau_x\right)X_f^k$\label{DAPDG:line:2}
            
            \State $Y_g^k = \tau_y Y^k + \left(1 - \tau_y\right)Y_f^k$\label{DAPDG:line:3}
            
            \State $U^{k+1}=X^k+\eta_x\alpha_x\left(X_g^k - X^k\right) - \eta_x\beta_x A^T\left(A X^k-\nabla G_k\right) - \eta_x\left(\nabla F_k + A^T Y_m^k\right)$\label{DAPDG:line:4}
            
            \State $W^{k+1} = Y^k + \eta_y\alpha_y\left(Y_g^k - Y^k\right) - \eta_y\beta_y A\left(A^T Y^k+\nabla F_k\right) - \eta_y\left(\nabla G_k - A U^{k+1}\right)$\label{DAPDG:line:5}
            
            \State $X^{k+1} = \text{Consensus}(U^{k+1}, T^k)$\label{DAPDG:line:6}
            
            \State $Y^{k+1} = \text{Consensus}(W^{k+1}, T^k)$\label{DAPDG:line:7}
            
            \State $X_f^{k+1} = X_g^k + \sigma_x\left(X^{k+1} - X^k\right)$\label{DAPDG:line:8}
            
            \State $Y_f^{k+1} = Y_g^k + \sigma_y\left(Y^{k+1} - Y^k\right)$\label{DAPDG:line:9}
            
            \EndFor
        \end{algorithmic}
    \end{algorithm}

    \begin{algorithm}[H]\caption{Consensus}
        \begin{algorithmic}
            \State \textbf{Input:} $X^0$, the number of iterations $T$
            
            \For{$t=0,1,\dots, T-1$}
            \State $X^{t+1}=X^t W^t$
            \EndFor
        \end{algorithmic}
    \end{algorithm}
    
    One may notice that it is possible to rewrite this algorithm in terms of average values of variables at nodes (algorithm is provided in appendix \ref{app:consensus_iterations}).
    We make Consensus iterations in order to make calls of stochastic oracle from relatively close points, allowing average gradient by nodes to be inexact gradient at average point with sufficiently small $\delta$. Consensus iterations do not change the average value due to doubly stochastic property of mixing matrices.
    \subsection{Consensus}
    We consider a sequence of  non-directed communication graphs $\left\{\left(V,E^k \right)\right\}_{k=0}^{\infty}$ and a sequence of corresponding mixing matrices $\left\{ W^k \right\}_{k=0}^{\infty}$ associated with it. We impose the following assumption.
    \begin{assumption}\label{ass:7}
        Mixing matrix sequence $\left\{ W^k \right\}_{k=0}^{\infty}$ satisfies the following properties.
        
        (Decentralized property) If $(i,j)\not\in E^k$, then $\left[ W^k \right]_{ij} = 0$.
        
        (Doubly stochastic property) $W^k 1_n = 1_n$, $1_n^T W^k = 1_n^T$.
        
        (Contraction property) There exist $\tau \in \mathbb{Z}_{++}$ and $\lambda \in \left(0, 1\right)$, such that for every $k \geq \tau - 1$ it holds
        $$\|W_{\tau}^k X - \overline{X}\| \leq (1-\lambda)\|X-\overline{X}\|,$$
        where $W_{\tau}^k = W^k\dots W^{k - \tau + 1}$.
    \end{assumption}
    
    \subsection{Complexity results for Algorithm \ref{alg:DAPDG}}
    \noindent
    
    Our analysis shows that
    performance of Algorithm \ref{alg:DAPDG} depends on global constants.
    
    \begin{theorem}
        Let $f_i(x)$ be $\mu_{x, i}$-strongly convex and $L_{x, i}$ smooth, and let $g_i(y)$ be $\mu_{y, i}$-strongly convex and $L_{y, i}$ smooth ($\mu_{x, i},\mu_{y, i}\geq0$). Also assume that there exist $i_1$ and $i_2$ such that $\mu_{x,i_1},\mu_{y,i_2}>0$. The $i$'th node has access only to unbiased stochastic gradients of function $f_i$ with variance $\sigma_{f, i}^2$ and of function $g_i$ with variance $\sigma_{g, i}^2$. Also suppose Assumption \ref{ass:7} holds. There exist $L_{xy}$, $\mu_{xy}$ and $\mu_{yx}$, which satisfy Assumption \ref{ass:5} for pairs of functions $(f_i, g_i)$ at every node. Then there exist sets of constants of Algorithm \ref{alg:DAPDG}, such that Algorithm~\ref{alg:DAPDG} has the following complexities.
        
        The number of iterations of Algorithm \ref{alg:DAPDG}
        $$N = \mathcal{O}\left(  \frac{1}{1-\theta} \log\left(\frac{D''}{\epsilon}\right)\right).$$
        
        The number of communications
        $$N_{comm} = \mathcal{O}\left( \frac{1}{1-\theta}\kappa\log\left(\frac{D''}{\epsilon}\right)\log\left(\frac{D'}{\epsilon}\right)\right).$$
        
        The number of stochastic oracle calls at node $i$
        $$N_{comp}^i  = 2N + \mathcal{O}\left( \frac{\max\{\omega, \omega^{-1}\}}{nL_{xy}(1-\theta)^2\epsilon}\left( \left( \frac{1}{L_{x}}+\frac{\omega}{L_{xy}} \right)\sigma_{f, i}^2 + \left( \frac{1}{L_{y}}+\frac{1}{L_{xy}\omega} \right)\sigma_{g, i}^2 \right)\log\left(\frac{D''}{\epsilon}\right)\right).$$
        
        Here is the list of possible estimations depending on different constants:
        \begin{align*}
            \frac{1}{1-\theta} &=
            \mathcal{O}\left(\max\left\{\sqrt{\frac{L_{x}}{\mu_{x}}}, \sqrt{\frac{L_{y}}{\mu_{y}}},\frac{L_{xy}}{\sqrt{\mu_{x}\mu_{y}}}\right\}\right),~ \omega = \sqrt{\frac{\mu_{y}}{\mu_{x}}},\\
            \frac{1}{1-\theta} & =
            \mathcal{O}\left(\max\left\{
            \frac{\sqrt{L_{x}L_{y}}}{\mu_{xy}},
            \frac{L_{xy}}{\mu_{xy}}\sqrt{\frac{L_{x}}{\mu_{x}}},
            \frac{L_{xy}^2}{\mu_{xy}^2}
            \right\}\right),~ \omega = \sqrt{\frac{\mu_{xy}^2}{2\mu_{x}L_{x}}},\\
            \frac{1}{1-\theta} & =
            \mathcal{O}\left(\max\left\{
            \frac{\sqrt{L_{x}L_{y}}}{\mu_{yx}},
            \frac{L_{xy}}{\mu_{yx}}\sqrt{\frac{L_{y}}{\mu_{y}}},
            \frac{L_{xy}^2}{\mu_{yx}^2}
            \right\}\right),~ \omega = \sqrt{\frac{2\mu_{y}L_{y}}{\mu_{yx}^2}},\\
            \frac{1}{1-\theta} &= \mathcal{O}\left(\max\left\{
            \frac{\sqrt{L_{x}L_{y}}L_{xy}}{\mu_{xy}\mu_{yx}},
            \frac{L_{xy}^2}{\mu_{yx}^2},
            \frac{L_{xy}^2}{\mu_{xy}^2}
            \right\}\right),~ \omega = \frac{\mu_{xy}}{\mu_{yx}}\sqrt{\frac{L_{y}}{L_{x}}},
        \end{align*}
        
        where $D'$ and $D''$ are polynomial and not depend on $\epsilon$.
    \end{theorem}
    The proof is provided in Appendix \ref{app:consensus_iterations}.
    
    The result associated with the number of communication calls is provided under the time-varying graph conditions. However, in static graph setting we can use Chebyshev-Accelerated subroutine \cite{ScamanLower} to get $\sqrt{\kappa}$ dependence from the characteristic number of the network. Therefore, our result is optimal up to the logarithmic factor according to \cite{kovalev2022optimal} (when $L_x=L_y=L_{xy}$ and $\mu_x=\mu_y$). Our stochastic oracle calls result is also optimal according to \cite{beznosikov2020distributed}. The deterministic side for coupling term is optimal (up to logarithmic factor for communications) according to \cite{rogozin2021decentralized}.
    
    It is worth mentioning, that this result only serves the case of $\mu_x,\mu_y>0$ due to specificity of analysis of basis algorithm. The Lemma~\ref{lemma:3} requires non-zero $\mu_g$, therefore if we wanted to extend our result to convex-concave case, we would need to use regularization trick, this would be possible with duality gap criterion, not in our case.
    
    The work of A. Gasnikov was supported by a grant for research centers in the field of artificial intelligence, provided by the Analytical Center for the Government of the Russian Federation in accordance with the subsidy agreement (agreement identifier 000000D730321P5Q0002) and the agreement with the Ivannikov Institute for System Programming of the Russian Academy of Sciences dated November 2, 2021 No. 70-2021-00142.
    
	% use references.bib, not \begin{thebibliography}
		\bibliographystyle{abbrv}
		\bibliography{references}
		% 	\begin{thebibliography}{99}
			
			% 		\bibitem{First-Order-Methods-of-Smooth-Convex-Optimization-with-Inexact-Oracle}
			% 		\textit{O.~Devolder, Fran¸cois Glineur, Y.Nesterov}
			% 		First-Order Methods of Smooth Convex Optimization with Inexact Oracle// Mathematical Programming – 2014.
			
			% 		\bibitem{decentr2009}
			% 		\textit{A.Rogozin, V.Lukoshkin, A.Gasnikov, D.Kovalev, E.Shulgin}
			% 		Towards accelerated rates for distributed optimization over time- varying networks// Lecture Notes in Computer Science 2021 V.13078.
			
			% 		\bibitem{minimax2112}
			% 		\textit{D.~Kovalev, A.~Gasnikov, Peter Richtárik}
			% 		Accelerated Primal-Dual Gradient Method for Smooth and Convex-Concave Saddle-Point Problems with Bilinear Coupling// 2022.
			
			% 		\bibitem{similar_determ2201}
			% 		\textit{K.K.Thekumparampil, Niao He, Sewoong Oh}
			% 		Lifted Primal-Dual Method for Bilinearly Coupled Smooth Minimax Optimization// 2022.
			
			% 		\bibitem{couplin_term2111}
			% 		\textit{Xuan Zhang, N. S. Aybat, Mert Gurbuzbalaban}
			% 		Robust Accelerated Primal-Dual Methods for Computing Saddle Points// 2021.
			
			% 		\bibitem{stoch_decentr2103}
			% 		\textit{A. Rogozin, M. Bochko, P. Dvurechensky, A. Gasnikov, V. Lukoshkin}
			% 		An accelerated method for decentralized distributed stochastic opti- mization over time-varying graphs//2021 60th IEEE Conference on Decision and Control 2021.
			
			% 		\bibitem{decent_review2011}
			% 		\textit{E.Gorbunov, A.Rogozin, A.Beznosikov, D.Dvinskikh, A.Gasnikov}
			% 		Recent theoretical advances in decentralized distributed convex optimization// 2021.
			% 	\end{thebibliography}

		\newpage
		\appendix
		{\large\textbf{Supplementary material}}
		\section{Inexact setting}\label{app:inexact_setting}
		
		\subsection{Proof of the Theorem~\ref{th4}}
		
		\noindent
		
		\begin{lemma}
			Let us introduce several definitions.
			\begin{equation}
				\tau_x = (\sigma_x^{-1}+1/2)^{-1},\label{eq:1}
			\end{equation}
			\begin{equation}
				\alpha_x=\mu_x,
			\end{equation}
			\begin{equation}
				\beta_x=\min\left\{ \frac{1}{2L_y}, \frac{1}{2\eta_x L_{xy}^2} \right\}.
			\end{equation}
			Then the following inequality holds
			
			\begin{equation}
				\begin{split}
					\frac{1}{\eta_x}\mathbb{E}_{\xi_x^k, \xi_y^k}\|x^{k+1} - x^*\|^2 &\leq \left( \frac{1}{\eta_x}-\mu_x-\beta_x \mu_{yx}^2 \right)\|x^k - x^*\|^2
					\\&
					+\left( \mu_x + L_x\sigma_x - \frac{1}{2\eta_x} \right)\mathbb{E}_{\xi_x^k, \xi_y^k}\|x^{k+1} - x^k\|^2 
					\\&
					+B_g(y_g^k,y^*)
					-B_f(x_g^k,x^*)
					-\frac{2}{\sigma_x}\E_{\xi_x^k\,\xi_y^k}B_f(x_f^{k+1},x^*)
					\\&
					+\left( \frac{2}{\sigma_x} - 1 \right)B_f(x_f^k,x^*)
					-2\mathbb{E}_{\xi_x^k, \xi_y^k}\langle A^T(y_m^k-y^*), x^{k+1} - x^* \rangle
					\\&
					+\delta_y + \left( \frac{4}{\sigma_x}+1 \right)\delta_x + \beta_x\sigma_g^2 + 2\eta_x\sigma_f^2.
				\end{split}\label{L7}
			\end{equation}
		\end{lemma}
		
		\begin{proof}
			
			The proof is similar to the proof of Lemma B.2 in \hyperlink{2}{[2]}. But we should expand the proof to cover inexact stochastic case. In order to do this we need replace several inequalities in the proof from above article on corresponding inequalities with inexact stochastic oracle.
			
			In the following analysis, we will need the following inequalities (the term $\xi^{k-1}$ is omitted).
			
			\begin{equation}
				\frac{1}{2L_y}\mathbb{E}_{\xi_y^k}\|\nabla g_{\delta}(y_g^k, \xi_y^k)-\nabla g(y^*)\|^2 \leq B_g(y_g^k,y^*)+\delta_y+\frac{\sigma_g^2}{2L_y},\label{diff:1}
			\end{equation}
			
			\begin{equation}
				\mathbb{E}_{\xi_x^k}\langle \nabla f_{\delta}(x_g^k, \xi_x^k) - \nabla f(x^*), x^{k+1} - x^* \rangle \geq \E_{\xi_x^k}\langle \nabla f_{\delta}(x_g^k) - \nabla f(x^*), x^{k+1} - x^* \rangle - \eta_x\sigma_f^2,\label{diff:2}
			\end{equation}
			
			\begin{equation}
				\langle \nabla f_{\delta}(x^g_k) - \nabla f(x^*), x_f^{k+1} - x_g^k \rangle \geq B_f(x_f^{k+1},x^*) - B_f(x_g^k,x^*)-\frac{L_x}{2}\|x_f^{k+1}-x_g^k\|^2-\delta_x,\label{diff:3}
			\end{equation}

			\begin{equation}
				2\langle \nabla f_{\delta}(x_g^k) - \nabla f(x^*), x_g^k - x^* \rangle \geq 2B_f(x_g^k,x^*) + \mu_x\|x_g^k-x^*\|^2 - 2\delta_x,\label{diff:4}
			\end{equation}
			
			\begin{equation}
				\langle \nabla f_{\delta}(x^g_k) - \nabla f(x^*), x_f^k - x_g^k \rangle \leq B_f(x_f^k,x^*)-B_f(x_g^k, x^*)+\delta_x.\label{diff:5}
			\end{equation}
			
			Let us prove these inequalities.
			
			Inequality \eqref{diff:1}. Using $\mathbb{E}_{\xi_y^k}\nabla g_{\delta}(y_g^k, \xi_y^k) = \nabla g_{\delta}(y_g^k)$ and according to Theorem 1 from \cite{First-Order-Methods-of-Smooth-Convex-Optimization-with-Inexact-Oracle}, we have
			\begin{align*}
				\frac{1}{2L_y}\mathbb{E}_{\xi_y^k}\|\nabla g_{\delta}(y_g^k, \xi_y^k)-\nabla g(y^*)\|^2 &= \frac{1}{2L_y}\mathbb{E}_{\xi_y^k}\|\nabla g_{\delta}(y_g^k, \xi_y^k)-\nabla g_{\delta}(y_g^k)\|^2
				\\&
				+ \frac{1}{2L_y}\|\nabla g_{\delta}(y_g^k)-\nabla g(y^*)\|^2 \\& \leq 
				g(y_g^k)-g(y^*)-\langle \nabla g(y^*), y_g^k - y^* \rangle + \delta_y + \frac{\sigma_g^2}{2L_y}
				\\&= B_g(y_g^k,y^*)+\delta_y+\frac{\sigma_g^2}{2L_y}.
			\end{align*}
			
			We choose inexact oracle $(\hat{g}_{\delta}, \nabla \hat{g}_{\delta})$ such that it is the same as $(g_{\delta}, \nabla g_{\delta})$ at all points except $y^*$, and at $y^*$ it equals $(g(y^*), \nabla g(y^*))$.
			
			Inequality \eqref{diff:2}:
			\begin{align*}
				\mathbb{E}_{\xi_x^k}\langle \nabla f_{\delta}(x_g^k, \xi_x^k) - \nabla f(x^*), x^{k+1} - x^* \rangle &= \mathbb{E}_{\xi_x^k}\langle \nabla f_{\delta}(x_g^k, \xi_x^k) -  \nabla f_{\delta}(x_g^k), x^{k+1} - x^* \rangle 
				\\&+ \langle \nabla f_{\delta}(x_g^k) - \nabla f(x^*), x^{k+1} - x^* \rangle.
			\end{align*}
			
			Using Line~\ref{ISAPDG:line:4} of the Algorithm~\ref{alg:ISAPDG} we obtain
			\begin{align*}
				\mathbb{E}_{\xi_x^k}\langle \nabla f_{\delta}(x_g^k, \xi_x^k) - \nabla f(x^*), x^{k+1} - x^* \rangle &= \mathbb{E}_{\xi_x^k}\langle \nabla f_{\delta}(x_g^k, \xi_x^k) -  \nabla f_{\delta}(x_g^k), -\eta_x  \nabla f_{\delta}(x_g^k, \xi_x^k) \rangle 
				\\&+ \E_{\xi_x^k}\langle \nabla f_{\delta}(x_g^k) - \nabla f(x^*), x^{k+1} - x^* \rangle 
				\\&= \E_{\xi_x^k}\langle \nabla f_{\delta}(x_g^k) - \nabla f(x^*), x^{k+1} - x^* \rangle 
				\\&- \eta_x\E_{\xi_x^k}\|\nabla f_{\delta}(x_g^k, \xi_x^k) - \nabla f_{\delta}(x_g^k)\|^2 
				\\&\geq  \E_{\xi_x^k} \langle \nabla f_{\delta}(x_g^k) - \nabla f(x^*), x^{k+1} - x^* \rangle -\eta_x \sigma_f^2.
			\end{align*}
			
			Inequality \eqref{diff:3}:
			\begin{align*}
				\langle \nabla f_{\delta}(x^g_k) - \nabla f(x^*), x_f^{k+1} - x_g^k \rangle
				&\geq f(x_f^{k+1}) - f_{\delta}(x_g^k)-\frac{L_x}{2}\|x_f^{k+1}-x_g^k\|^2
				\\&-\delta_x - \langle \nabla f(x^*), x_f^{k+1} - x_g^k \rangle 
				\\&
				\geq f(x_f^{k+1}) - f(x_g^k)-\frac{L_x}{2}\|x_f^{k+1}-x_g^k\|^2
				\\&-\delta_x - \langle \nabla f(x^*), x_f^{k+1} - x_g^k \rangle
				\\& = B_f(x_f^{k+1},x^*) - B_f(x_g^k,x^*)-\frac{L_x}{2}\|x_f^{k+1}-x_g^k\|^2-\delta_x.
			\end{align*}
			Inequality \eqref{diff:4}:
			\begin{align*}
				2(f(x^*)-f(x_g^k))-2\langle \nabla f_{\delta}(x_g^k), x^*-x_g^k \rangle + 2\delta_x &\geq 2(f(x^*)-f_{\delta}(x_g^k))-2\langle \nabla f_{\delta}(x_g^k), x^*-x_g^k \rangle
				\\&
				\geq \mu_x\|x_g^k-x^*\|^2.
			\end{align*}
			
			Inequality \eqref{diff:5}:
			\begin{align*}
				\langle \nabla f_{\delta}(x^g_k) - \nabla f(x^*), x_f^k - x_g^k \rangle &\leq f(x_f^k) - f_{\delta}(x_g^k)-\langle \nabla f(x^*), x_f^k - x_g^k \rangle
				\\&
				\leq f(x_f^k) - f(x_g^k)-\langle \nabla f(x^*), x_f^k - x_g^k \rangle+\delta_x 
				\\&= B_f(x_f^k,x^*)-B_f(x_g^k)+\delta_x.
			\end{align*}
			
			Using Line~\ref{ISAPDG:line:4} of the Algorithm~\ref{alg:ISAPDG} we get
			\begin{align*}
				\frac{1}{\eta_x}\sqn{x^{k+1} - x^*}
				&=
				\frac{1}{\eta_x}\sqn{x^{k} - x^*} + \frac{2}{\eta_x}\<x^{k+1} - x^k,x^{k+1} - x^*> - \frac{1}{\eta_x}\sqn{x^{k+1} - x^k}
				\\&=
				\frac{1}{\eta_x}\sqn{x^{k} - x^*} + 2\alpha_x\<x_g^k - x^k,x^{k+1}- x^*>
				\\&-
				2\beta_x\<\mA^\top(\mA x^k - \nabla g_{\delta}(y_g^k, \xi_y^k, \xi^{k-1}),x^{k+1} - x^*>
				\\&-
				2\<\nabla f_{\delta}(x_g^k, \xi_x^k, \xi^{k-1}) + \mA^\top y_m^k,x^{k+1} - x^*>
				-
				\frac{1}{\eta_x}\sqn{x^{k+1} - x^k}.
			\end{align*}
			Using the parallelogram rule we get
			\begin{align*}
				\frac{1}{\eta_x}\sqn{x^{k+1} - x^*}
				&=
				\frac{1}{\eta_x}\sqn{x^{k} - x^*}
				\\&+
				\alpha_x\left(\sqn{x_g^k - x^*}  - \sqn{x_g^k - x^{k+1}} - \sqn{x^{k} - x^*}+\sqn{x^{k+1} - x^k}\right)
				\\&-
				2\beta_x\<\mA x^k - \nabla g_{\delta}(y_g^k, \xi_y^k, \xi^{k-1}),\mA(x^{k+1} - x^*)>
				\\&-
				2\<\nabla f_{\delta}(x_g^k, \xi_x^k, \xi^{k-1}) + \mA^\top y_m^k,x^{k+1} - x^*>
				-
				\frac{1}{\eta_x}\sqn{x^{k+1} - x^k}.
			\end{align*}
			Using the optimality condition $\nabla g(y^*) = \mA x^*$, which follows from $\nabla_y F(x^*, y^*) = 0$, and the parallelogram rule we get
			\begin{align*}
				\frac{1}{\eta_x}\sqn{x^{k+1} - x^*}
				&=
				\frac{1}{\eta_x}\sqn{x^{k} - x^*}
				\\&+
				\alpha_x\left(\sqn{x_g^k - x^*}  - \sqn{x_g^k - x^{k+1}} - \sqn{x^{k} - x^*}+\sqn{x^{k+1} - x^k}\right)
				\\&+
				\beta_x\left(\sqn{\mA(x^{k+1} - x^k)}  - \sqn{\mA(x^k - x^*)}\right) 
				\\&+
				\beta_x\left(\sqn{\nabla g_{\delta}(y_g^k, \xi_y^k, \xi^{k-1}) - \nabla g(y^*)} - \sqn{\nabla g_{\delta}(y_g^k, \xi_y^k, \xi^{k-1}) - \mA(x^{k+1})}\right)
				\\&-
				2\<\nabla f_{\delta}(x_g^k, \xi_x^k, \xi^{k-1}) + \mA^\top y_m^k,x^{k+1} - x^*>
				-
				\frac{1}{\eta_x}\sqn{x^{k+1} - x^k}.
			\end{align*}
			Using Assumption~\ref{ass:5} and Equation~\ref{diff:1}, we get
			\begin{align*}
				\frac{1}{\eta_x}\E_{\xi_x^k, \xi_y^k}\sqn{x^{k+1} - x^*}
				&\leq
				\frac{1}{\eta_x}\sqn{x^{k} - x^*}
				+
				\alpha_x\sqn{x_g^k - x^*}  - \alpha_x\sqn{x^{k} - x^*}
				\\&+
				\alpha_x\E_{\xi_x^k, \xi_y^k}\sqn{x^{k+1} - x^k}
				+
				\beta_xL_{xy}^2\E_{\xi_x^k, \xi_y^k}\sqn{x^{k+1} - x^k}
				\\&-
				\beta_x\mu_{yx}^2\sqn{x^k - x^*}
				+
				2\beta_xL_y\bg_g(y_g^k,y^*) + 2\beta_xL_y\delta_y(\xi^{k-1})+\beta_x\sigma_g^2
				\\&-
				2\E_{\xi_x^k, \xi_y^k}\<\nabla f_{\delta}(x_g^k, \xi_x^k, \xi^{k-1}) + \mA^\top y_m^k,x^{k+1} - x^*>
				\\&-
				\frac{1}{\eta_x}\E_{\xi_x^k, \xi_y^k}\sqn{x^{k+1} - x^k}
				=
				\left(\frac{1}{\eta_x} - \alpha_x- \beta_x\mu_{yx}^2\right)\sqn{x^{k} - x^*}
				\\&+
				\left(\beta_xL_{xy}^2 + \alpha_x-\frac{1}{\eta_x}\right)
				\E_{\xi_x^k, \xi_y^k}\sqn{x^{k+1} - x^k}
				+
				2\beta_xL_y\bg_g(y_g^k,y^*)
				\\&+
				\alpha_x\sqn{x_g^k - x^*}
				-
				2\E_{\xi_x^k, \xi_y^k}\<\nabla f_{\delta}(x_g^k, \xi_x^k, \xi^{k-1}) + \mA^\top y_m^k,x^{k+1} - x^*> 
				\\&+
				2\beta_xL_y\delta_y(\xi^{k-1})+\beta_x\sigma_g^2.
			\end{align*}
			Using the optimality condition $\nabla f(x^*) + \mA^\top y^* = 0$, which follows from $\nabla_x F(x^*, y^*) = 0$ and Equation~\ref{diff:2} , we get
			\begin{align*}
				\frac{1}{\eta_x}\E_{\xi_x^k, \xi_y^k}\sqn{x^{k+1} - x^*}
				&\leq
				\left(\frac{1}{\eta_x} - \alpha_x - \beta_x\mu_{yx}^2\right)\sqn{x^{k} - x^*}
				\\&+
				\left(\beta_xL_{xy}^2 + \alpha_x-\frac{1}{\eta_x}\right)
				\E_{\xi_x^k, \xi_y^k}\sqn{x^{k+1} - x^k}
				+
				2\beta_xL_y\bg_g(y_g^k,y^*)
				\\&+
				\alpha_x\sqn{x_g^k - x^*}
				-
				2\E_{\xi_x^k, \xi_y^k}\<\nabla f_{\delta}(x_g^k, \xi_x^k, \xi^{k-1}) - \nabla f(x^*),x^{k+1} - x^*>
				\\&-
				2\E_{\xi_x^k, \xi_y^k}\<\mA^\top (y_m^k - y^*),x^{k+1} - x^*>
				+
				2\beta_xL_y\delta_y(\xi^{k-1})+\beta_x\sigma_g^2
				\\&=
				\left(\frac{1}{\eta_x} - \alpha_x - \beta_x\mu_{yx}^2\right)\sqn{x^{k} - x^*}
				\\&+
				\left(\beta_xL_{xy}^2 + \alpha_x-\frac{1}{\eta_x}\right)
				\E_{\xi_x^k, \xi_y^k}\sqn{x^{k+1} - x^k}
				\\&+
				2\beta_xL_y\bg_g(y_g^k,y^*)
				+
				\alpha_x\sqn{x_g^k - x^*}
				\\&-
				2\E_{\xi_x^k, \xi_y^k}\<\nabla f_{\delta}(x_g^k, \xi^{k-1}) - \nabla f(x^*),x^{k+1}- x^k + x^k - x_g^k + x_g^k - x^*>
				\\&-
				2\E_{\xi_x^k, \xi_y^k}\<\mA^\top (y_m^k - y^*),x^{k+1} - x^*> + 2\beta_xL_y\delta_y(\xi^{k-1})+\beta_x\sigma_g^2+2\eta_x\sigma_f^2.
			\end{align*}
			Using $\mu_y$-strong convexity of $f$ and Lines~\ref{ISAPDG:line:2} and ~\ref{ISAPDG:line:6} of the Algorithm~\ref{alg:ISAPDG} and Equation~\ref{diff:4} we get
			\begin{align*}
				\frac{1}{\eta_x}\E_{\xi_x^k, \xi_y^k}\sqn{x^{k+1} - x^*}
				&\leq
				\left(\frac{1}{\eta_x} - \alpha_x - \beta_x\mu_{yx}^2\right)\sqn{x^{k} - x^*}
				\\
				&\quad+
				\left(\beta_xL_{xy}^2 + \alpha_x-\frac{1}{\eta_x}\right)
				\E_{\xi_x^k, \xi_y^k}\sqn{x^{k+1} - x^k}
				+
				2\beta_xL_y\bg_g(y_g^k,y^*)
				\\
				&\quad+
				\alpha_x\sqn{x_g^k - x^*}
				-
				\frac{2}{\sigma_x}\E_{\xi_x^k, \xi_y^k}\<\nabla f_{\delta}(x_g^k, \xi^{k-1}) - \nabla f(x^*),x_f^{k+1}- x_g^k>
				\\
				&\quad+
				\frac{2(1-\tau_x)}{\tau_x}\<\nabla f_{\delta}(x_g^k, \xi^{k-1}) - \nabla f(x^*),x_f^{k}- x_g^k>
				-
				2\bg_f(x_g^k,x^*)
				\\
				&\quad-
				\mu_x\sqn{x_g^k -x^*} + 2\delta_x(\xi^{k-1})
				-
				2\E_{\xi_x^k, \xi_y^k}\<\mA^\top (y_m^k - y^*),x^{k+1} - x^*> 
				\\
				&\quad+
				2\beta_xL_y\delta_y(\xi^{k-1})+\beta_x\sigma_g^2+2\eta_x\sigma_f^2
				\\
				&=
				\left(\frac{1}{\eta_x} - \alpha_x - \beta_x\mu_{yx}^2\right)\sqn{x^{k} - x^*}
				\\
				&\quad+
				\left(\beta_xL_{xy}^2 + \alpha_x-\frac{1}{\eta_x}\right)
				\E_{\xi_x^k, \xi_y^k}\sqn{x^{k+1} - x^k}
				+
				(\alpha_x-\mu_x)\sqn{x_g^k - x^*}
				\\
				&\quad+
				2\beta_xL_y\bg_g(y_g^k,y^*)
				-
				2\bg_f(x_g^k,x^*) 
				\\
				&\quad-
				\frac{2}{\sigma_x}\E_{\xi_x^k, \xi_y^k}\<\nabla f_{\delta}(x_g^k, \xi^{k-1}) - \nabla f(x^*),x_f^{k+1}- x_g^k>
				\\
				&\quad+
				\frac{2(1-\tau_x)}{\tau_x}\<\nabla f_{\delta}(x_g^k, \xi^{k-1}) - \nabla f(x^*),x_f^{k}- x_g^k>
				\\
				&\quad-
				2\E_{\xi_x^k, \xi_y^k}\<\mA^\top (y_m^k - y^*),x^{k+1} - x^*>
				\\
				&\quad+2\beta_xL_y\delta_y(\xi^{k-1})+\beta_x\sigma_g^2+2\eta_x\sigma_f^2+ 2\delta_x(\xi^{k-1}).
			\end{align*}
			Using Equation~\ref{diff:5}, we get
			\begin{align*}
				\frac{1}{\eta_x}\E_{\xi_x^k, \xi_y^k}\sqn{x^{k+1} - x^*}
				&\leq
				\left(\frac{1}{\eta_x} - \alpha_x - \beta_x\mu_{yx}^2\right)\sqn{x^{k} - x^*}
				\\&+
				\left(\beta_xL_{xy}^2 + \alpha_x-\frac{1}{\eta_x}\right)
				\E_{\xi_x^k, \xi_y^k}\sqn{x^{k+1} - x^k}
				+
				(\alpha_x-\mu_x)\sqn{x_g^k - x^*}
				\\&+
				2\beta_xL_y\bg_g(y_g^k,y^*)
				-
				2\bg_f(x_g^k,x^*) 
				\\&-
				\frac{2}{\sigma_x}\E_{\xi_x^k, \xi_y^k}\<\nabla f_{\delta}(x_g^k, \xi^{k-1}) - \nabla f(x^*),x_f^{k+1}- x_g^k>
				\\&+
				\frac{2(1-\tau_x)}{\tau_x}\left(\bg_f(x_f^k,x^*) - \bg_f(x_g^k,x^*)\right)
				\\&-
				2\E_{\xi_x^k, \xi_y^k}\<\mA^\top (y_m^k - y^*),x^{k+1} - x^*>
				\\&+2\beta_xL_y\delta_y(\xi^{k-1})+\beta_x\sigma_g^2+2\eta_x\sigma_f^2+ 2\delta_x(\xi^{k-1})+\frac{2(1-\tau_x)}{\tau_x}\delta_x(\xi^{k-1}).
			\end{align*}
			Using Equation~\ref{diff:3}, we get
			\begin{align*}
				\frac{1}{\eta_x}\E_{\xi_x^k, \xi_y^k}\sqn{x^{k+1} - x^*}
				&\leq
				\left(\frac{1}{\eta_x} - \alpha_x - \beta_x\mu_{yx}^2\right)\sqn{x^{k} - x^*}
				\\&+
				\left(\beta_xL_{xy}^2 + \alpha_x-\frac{1}{\eta_x}\right)
				\E_{\xi_x^k, \xi_y^k}\sqn{x^{k+1} - x^k}
				+
				(\alpha_x-\mu_x)\sqn{x_g^k - x^*}
				\\&+
				2\beta_xL_y\bg_g(y_g^k,y^*)
				-
				2\bg_f(x_g^k,x^*) 
				\\&-
				\frac{2}{\sigma_x}\E_{\xi_x^k, \xi_y^k}\left(\bg_f(x_f^{k+1},x^*) - \bg_f(x_g^k,x^*) - \frac{L_x}{2}\sqn{x_f^{k+1} - x_g^k}\right)
				\\&+
				\frac{2(1-\tau_x)}{\tau_x}\left(\bg_f(x_f^k,x^*) - \bg_f(x_g^k,x^*)\right)
				\\&-
				2\E_{\xi_x^k, \xi_y^k}\<\mA^\top (y_m^k - y^*),x^{k+1} - x^*>
				\\&+2\beta_xL_y\delta_y(\xi^{k-1})+\beta_x\sigma_g^2+2\eta_x\sigma_f^2+\left(\frac{2}{\tau_x}+\frac{2}{\sigma_x}\right)\delta_x(\xi^{k-1}).
			\end{align*}
			Using Line~\ref{ISAPDG:line:6} of the Algorithm~\ref{alg:ISAPDG} we get
			\begin{align*}
				\frac{1}{\eta_x}\E_{\xi_x^k, \xi_y^k}\sqn{x^{k+1} - x^*}
				&\leq
				\left(\frac{1}{\eta_x} - \alpha_x - \beta_x\mu_{yx}^2\right)\sqn{x^{k} - x^*}
				\\&+
				\left(\beta_xL_{xy}^2 + \alpha_x-\frac{1}{\eta_x}\right)
				\E_{\xi_x^k, \xi_y^k}\sqn{x^{k+1} - x^k}
				+
				(\alpha_x-\mu_x)\sqn{x_g^k - x^*}
				\\&+
				2\beta_xL_y\bg_g(y_g^k,y^*)
				-
				2\bg_f(x_g^k,x^*) 
				\\&-
				\frac{2}{\sigma_x}\E_{\xi_x^k, \xi_y^k}\left(\bg_f(x_f^{k+1},x^*) - \bg_f(x_g^k,x^*) - \frac{L_x\sigma_x^2}{2}\sqn{x^{k+1} - x^k}\right)
				\\&+
				\frac{2(1-\tau_x)}{\tau_x}\left(\bg_f(x_f^k,x^*) - \bg_f(x_g^k,x^*)\right)
				\\&-
				2\E_{\xi_x^k, \xi_y^k}\<\mA^\top (y_m^k - y^*),x^{k+1} - x^*>
				\\&+2\beta_xL_y\delta_y(\xi^{k-1})+\beta_x\sigma_g^2+2\eta_x\sigma_f^2+\left(\frac{2}{\tau_x}+\frac{2}{\sigma_x}\right)\delta_x(\xi^{k-1}).
			\end{align*}
			Transforming this inequality we get
			\begin{align*}
				\frac{1}{\eta_x}\E_{\xi_x^k, \xi_y^k}\sqn{x^{k+1} - x^*}
				&\leq \left(\frac{1}{\eta_x} - \alpha_x - \beta_x\mu_{yx}^2\right)\sqn{x^{k} - x^*}
				\\&+
				\left(\beta_xL_{xy}^2 + \alpha_x + L_x\sigma_x -\frac{1}{\eta_x}\right)
				\E_{\xi_x^k, \xi_y^k}\sqn{x^{k+1} - x^k}
				\\&+
				(\alpha_x-\mu_x)\sqn{x_g^k - x^*}
				+
				2\beta_xL_y\bg_g(y_g^k,y^*)
				+
				\left(\frac{2}{\sigma_x} - \frac{2}{\tau_x}\right)\bg_f(x_g^k,x^*) 
				\\&-
				\frac{2}{\sigma_x}\E_{\xi_x^k, \xi_y^k}\bg_f(x_f^{k+1},x^*) 
				+
				\left(\frac{2}{\tau_x} - 2\right)\bg_f(x_f^k,x^*) 
				\\&-
				2\E_{\xi_x^k, \xi_y^k}\<\mA^\top (y_m^k - y^*),x^{k+1} - x^*>
				\\&+2\beta_xL_y\delta_y(\xi^{k-1})+\beta_x\sigma_g^2+2\eta_x\sigma_f^2+\left(\frac{2}{\tau_x}+\frac{2}{\sigma_x}\right)\delta_x(\xi^{k-1}).
			\end{align*}
			Using the definition of $\tau_x$, $\alpha_x$ and $\beta_x$ we get
			\begin{align*}
				\frac{1}{\eta_x}\E_{\xi_x^k, \xi_y^k}\sqn{x^{k+1} - x^*}
				&\leq
				\left(\frac{1}{\eta_x} - \mu_x - \beta_x\mu_{yx}^2\right)\sqn{x^{k} - x^*}
				\\&+
				\left(\mu_x + L_x\sigma_x -\frac{1}{2\eta_x}\right)
				\E_{\xi_x^k, \xi_y^k}\sqn{x^{k+1} - x^k}
				\\&
				+
				\bg_g(y_g^k,y^*)
				-
				\bg_f(x_g^k,x^*) 
				\\&-
				\frac{2}{\sigma_x}\E_{\xi_x^k, \xi_y^k}\bg_f(x_f^{k+1},x^*) 
				+
				\left(\frac{2}{\sigma_x} - 1\right)\bg_f(x_f^k,x^*) 
				\\&-
				2\E_{\xi_x^k, \xi_y^k}\<\mA^\top (y_m^k - y^*),x^{k+1} - x^*>
				\\&+
				\delta_y(\xi^{k-1})+\left(\frac{4}{\sigma_x}+1 \right)\delta_x(\xi^{k-1})+\beta_x\sigma_g^2+2\eta_x\sigma_f^2.
			\end{align*}
			
		\end{proof}
		
		\begin{lemma}\label{L8}
			Let us introduce several definitions.
			\begin{equation}
				\tau_y = (\sigma_y^{-1}+1/2)^{-1},
			\end{equation}
			\begin{equation}
				\alpha_y=\mu_y,
			\end{equation}
			\begin{equation}
				\beta_y=\min\left\{ \frac{1}{2L_x}, \frac{1}{2\eta_y L_{xy}^2} \right\}.
			\end{equation}
			Then the following inequality holds
			
			\begin{equation}
				\begin{split}
					\frac{1}{\eta_y}\E_{\xi_x^k, \xi_y^k}\|y^{k+1} - y^*\|^2 &\leq \left( \frac{1}{\eta_y}-\mu_y-\beta_y \mu_{xy}^2 \right)\|y^k - y^*\|^2 
					\\&\quad+
					\left( \mu_y + L_y\sigma_y - \frac{1}{2\eta_y} \right)\E_{\xi_x^k, \xi_y^k}\|y^{k+1} - y^k\|^2
					\\&\quad+B_f(x_g^k,x^*)-B_g(y_g^k,y^*)-\frac{2}{\sigma_y}\mathbb{E}_{\xi_x^k, \xi_y^k}B_g(y_f^{k+1},y^*)
					\\&\quad+
					\left( \frac{2}{\sigma_y} - 1 \right)B_g(y_f^k,y^*)
					+2\E_{\xi_x^k, \xi_y^k}\langle A(x^{k+1}-x^*), y^{k+1} - y^* \rangle
					\\&\quad+
					\delta_x(\xi^{k-1}) + \left( \frac{4}{\sigma_y}+1 \right)\delta_y(\xi^{k-1}) + \beta_y\sigma_f^2 + 2\eta_y\sigma_g^2.
				\end{split}
			\end{equation}
		\end{lemma}
		\begin{proof}
			The proof is similar to the proof of the previous lemma.
		\end{proof}
		
		\begin{lemma}\label{L9}
			Let $\eta_x$ be defined as
			$$\eta_x = \min\left\{ \frac{1}{4(\mu_x + L_x\sigma_x)}, \frac{\omega}{4L_{xy}} \right\},$$
			and let $\eta_y$ be defined as
			$$\eta_y = \min\left\{ \frac{1}{4(\mu_y + L_y\sigma_y)}, \frac{1}{4L_{xy}\omega} \right\},$$
			where $\omega > 0$ is a parameter. Let $\theta$ be defined as
			$$\theta = \theta(\omega, \sigma_x, \sigma_y) = 1 - \max\{\rho_a(\omega, \sigma_x, \sigma_y),\rho_b(\omega, \sigma_x, \sigma_y),\rho_c(\omega, \sigma_x, \sigma_y),\rho_d(\omega, \sigma_x, \sigma_y)\},$$
			where %$\rho_a(\omega, \sigma_x, \sigma_y),\rho_b(\omega, \sigma_x, \sigma_y),\rho_c(\omega, \sigma_x, \sigma_y),\rho_d(\omega, \sigma_x, \sigma_y)$ are defined as
			
			$\rho_a(\omega, \sigma_x, \sigma_y) = \left[\max\left\{ \frac{4(\mu_x + L_x\sigma_x)}{\mu_x}, \frac{2}{\sigma_x}, \frac{4(\mu_y + L_y\sigma_y)}{\mu_{y}}, \frac{2}{\sigma_y}, \frac{4L_{xy}}{\mu_x\omega}, \frac{4L_{xy}\omega}{\mu_y} \right\} \right]^{-1},$
			
			$\rho_b(\omega, \sigma_x, \sigma_y) = \left[\max\left\{ \frac{4(\mu_x + L_x\sigma_x)}{\mu_x}, \frac{2}{\sigma_x}, \frac{8L_x(\mu_y + L_y\sigma_y)}{\mu_{xy}^2}, \frac{2}{\sigma_y}, \frac{2L_{xy}^2}{\mu_{xy}^2}, \frac{8L_x L_{xy}\omega}{\mu_{xy}^2},\frac{4L_{xy}}{\mu_x\omega} \right\} \right]^{-1},$
			
			$\rho_c(\omega, \sigma_x, \sigma_y) = \left[\max\left\{ \frac{4(\mu_y + L_y\sigma_y)}{\mu_y}, \frac{2}{\sigma_y}, \frac{8L_y(\mu_x + L_x\sigma_x)}{\mu_{yx}^2}, \frac{2}{\sigma_x}, \frac{2L_{xy}^2}{\mu_{yx}^2}, \frac{8L_y L_{xy}}{\mu_{yx}^2\omega},\frac{4L_{xy}\omega}{\mu_y} \right\} \right]^{-1},$
			
			$\rho_d(\omega, \sigma_x, \sigma_y) = \left[\max\left\{ \frac{8L_y(\mu_x + L_x\sigma_x)}{\mu_{yx}^2}, \frac{2}{\sigma_x}, \frac{8L_x(\mu_y + L_y\sigma_y)}{\mu_{xy}^2}, \frac{2}{\sigma_y}, \frac{8L_y L_{xy}}{\mu_{yx}^2\omega},\frac{8L_x L_{xy}\omega}{\mu_{xy}^2},\frac{2L_{xy}^2}{\mu_{yx}^2},\frac{2L_{xy}^2}{\mu_{xy}^2} \right\} \right]^{-1}.$
			
			Let $\Psi^k$ be the following Lyapunov function:
			
			\begin{equation}
				\begin{split}
					\Psi^k &= \frac{1}{\eta_x}\|x^k - x^*\|^2+\frac{1}{\eta_y}\|y^k - y^*\|^2+\frac{2}{\sigma_x}B_f(x_f^k, x^*)+\frac{2}{\sigma_y}B_g(y_f^k, y^*)
					\\&
					+\frac{1}{4\eta_y}\|y^k - y^{k-1}\|^2-2\langle y^k-y^{k-1}, A(x^k-x^*) \rangle.
				\end{split}
			\end{equation}
			Then, the following inequalities hold
			\begin{equation}
				\Psi^k \geq \frac{3}{4\eta_x}\|x^k - x^*\|^2 + \frac{1}{\eta_y}\|y^k - y^*\|^2,
			\end{equation}
			\begin{equation}
				\E\Psi^{k+1}\leq \theta\E\Psi^k + \frac{4}{1-\theta}\left(\delta_x+\delta_y\right) + \frac{1}{2}\left(\frac{1}{L_x} + \frac{\omega}{L_{xy}}\right)\sigma_f^2 + \frac{1}{2}\left(\frac{1}{L_y} + \frac{1}{L_{xy}\omega}\right)\sigma_g^2.
			\end{equation}
		\end{lemma}
		
		\begin{proof}
			The proof of this lemma is similar to the proof of Lemma B.4. in \cite{minimax2112}.
			
			After adding up \eqref{L9} and \eqref{L8} we get
			\begin{align*}
				\mathrm{(LHS)}
				&\leq
				\left(\frac{1}{\eta_x} - \mu_x - \beta_x\mu_{yx}^2\right)\sqn{x^{k} - x^*}
				+
				\left(\frac{1}{\eta_y} - \mu_y - \beta_y\mu_{xy}^2\right)\sqn{y^{k} - y^*}
				\\&+
				\left(\mu_x + L_x\sigma_x -\frac{1}{2\eta_x}\right)
				\E_{\xi_x^k, \xi_y^k}\sqn{x^{k+1} - x^k}
				\\&+
				\left(\mu_y + L_y\sigma_y -\frac{1}{2\eta_y}\right)
				\E_{\xi_x^k, \xi_y^k}\sqn{y^{k+1} - y^k}
				+		
				\left(\frac{2}{\sigma_x} - 1\right)\bg_f(x_f^k,x^*) 
				\\&+
				\left(\frac{2}{\sigma_y} - 1\right)\bg_g(y_f^k,y^*) 
				+
				2\E_{\xi_x^k, \xi_y^k}\<y^{k+1} - y_m^k,\mA(x^{k+1} - x^*)>
				\\&+
				\left( 2 + \frac{4}{\sigma_x}\right)\delta_x(\xi^{k-1})+\left(2 + \frac{4}{\sigma_y}\right)\delta_y(\xi^{k-1}) + (\beta_y + 2\eta_x)\sigma_f^2 + (\beta_x + 2\eta_y)\sigma_g^2,
			\end{align*}
			where $\mathrm{(LHS)}$ is given as
			\begin{align*}
				\mathrm{(LHS)}&= \frac{1}{\eta_x}\E_{\xi_x^k, \xi_y^k}\sqn{x^{k+1} - x^*} + \frac{1}{\eta_y}\E_{\xi_x^k, \xi_y^k}\sqn{y^{k+1} - y^*} 
				\\&+
				\frac{2}{\sigma_x}\E_{\xi_x^k, \xi_y^k}\bg_f(x_f^{k+1},x^*) +
				\frac{2}{\sigma_y}\E_{\xi_x^k, \xi_y^k}\bg_g(y_f^{k+1},y^*).
			\end{align*}
			Using Line~\ref{ISAPDG:line:2} of the Algorithm~\ref{alg:ISAPDG} we get
			\begin{align*}
				\mathrm{(LHS)}
				&\leq
				\left(\frac{1}{\eta_x} - \mu_x - \beta_x\mu_{yx}^2\right)\sqn{x^{k} - x^*}
				+
				\left(\frac{1}{\eta_y} - \mu_y - \beta_y\mu_{xy}^2\right)\sqn{y^{k} - y^*}
				\\&+
				\left(\mu_x + L_x\sigma_x -\frac{1}{2\eta_x}\right)
				\E_{\xi_x^k, \xi_y^k}\sqn{x^{k+1} - x^k}
				\\&+
				\left(\mu_y + L_y\sigma_y -\frac{1}{2\eta_y}\right)
				\E_{\xi_x^k, \xi_y^k}\sqn{y^{k+1} - y^k}
				\\&+		
				\left(\frac{2}{\sigma_x} - 1\right)\bg_f(x_f^k,x^*) 
				+
				\left(\frac{2}{\sigma_y} - 1\right)\bg_g(y_f^k,y^*) 
				\\&+
				2\E_{\xi_x^k, \xi_y^k}\<y^{k+1} -y^k,\mA(x^{k+1} - x^*)>
				-
				2\theta\<y^k -y^{k-1},\mA(x^{k+1} - x^*)>
				\\&+ \left( 2 + \frac{4}{\sigma_x}\right)\delta_x(\xi^{k-1})+\left(2 + \frac{4}{\sigma_y}\right)\delta_y(\xi^{k-1}) + (\beta_y + 2\eta_x)\sigma_f^2 + (\beta_x + 2\eta_y)\sigma_g^2.
			\end{align*}
			
			Using Assumption~\ref{ass:5} we get
			\begin{align*}
				\mathrm{(LHS)}
				&\leq
				\left(\frac{1}{\eta_x} - \mu_x - \beta_x\mu_{yx}^2\right)\sqn{x^{k} - x^*}
				+
				\left(\frac{1}{\eta_y} - \mu_y - \beta_y\mu_{xy}^2\right)\sqn{y^{k} - y^*}
				\\&+
				\left(\mu_x + L_x\sigma_x -\frac{1}{2\eta_x}\right)
				\E_{\xi_x^k, \xi_y^k}\sqn{x^{k+1} - x^k}
				\\&+
				\left(\mu_y + L_y\sigma_y -\frac{1}{2\eta_y}\right)
				\E_{\xi_x^k, \xi_y^k}\sqn{y^{k+1} - y^k}
				\\&+		
				\left(\frac{2}{\sigma_x} - 1\right)\bg_f(x_f^k,x^*) 
				+
				\left(\frac{2}{\sigma_y} - 1\right)\bg_g(y_f^k,y^*) 
				\\&+
				2\E_{\xi_x^k, \xi_y^k}\<y^{k+1} -y^k,\mA(x^{k+1} - x^*)>
				-
				2\theta\<y^k -y^{k-1},\mA(x^{k} - x^*)>
				\\&+
				2\theta L_{xy}\E_{\xi_x^k, \xi_y^k}\norm{y^k - y^{k-1}}\norm{x^{k+1} - x^k}
				\\&+ \left( 2 + \frac{4}{\sigma_x}\right)\delta_x(\xi^{k-1})+\left(2 + \frac{4}{\sigma_y}\right)\delta_y(\xi^{k-1}) + (\beta_y + 2\eta_x)\sigma_f^2 + (\beta_x + 2\eta_y)\sigma_g^2.
			\end{align*}
			Using the definition of $\eta_x$ and $\eta_y$ and the definition of $\theta< 1$ we get
			\begin{align*}
				\mathrm{(LHS)}
				&\leq
				\left(\frac{1}{\eta_x} - \mu_x - \beta_x\mu_{yx}^2\right)\sqn{x^{k} - x^*}
				+
				\left(\frac{1}{\eta_y} - \mu_y - \beta_y\mu_{xy}^2\right)\sqn{y^{k} - y^*}
				\\&-
				\frac{1}{4\eta_x}
				\E_{\xi_x^k, \xi_y^k}\sqn{x^{k+1} - x^k}
				-
				\frac{1}{4\eta_y}
				\E_{\xi_x^k, \xi_y^k}\sqn{y^{k+1} - y^k}
				\\&+		
				\left(\frac{2}{\sigma_x} - 1\right)\bg_f(x_f^k,x^*) 
				+
				\left(\frac{2}{\sigma_y} - 1\right)\bg_g(y_f^k,y^*) 
				\\&+
				2\E_{\xi_x^k, \xi_y^k}\<y^{k+1} -y^k,\mA(x^{k+1} - x^*)>
				-
				2\theta\<y^k -y^{k-1},\mA(x^{k} - x^*)>
				\\&+
				\frac{\theta }{2\sqrt{\eta_x\eta_y}}\E_{\xi_x^k, \xi_y^k}\norm{y^k - y^{k-1}}\norm{x^{k+1} - x^k}
				\\&+ \frac{4}{1-\theta}\left(\delta_x(\xi^{k-1})+\delta_y(\xi^{k-1})\right) + \left(\beta_y + \frac{\omega}{2L_{xy}}\right)\sigma_f^2 + \left(\beta_x + \frac{1}{2L_{xy}\omega}\right)\sigma_g^2.
			\end{align*}
			Transforming this inequality we get
			\begin{align*}
				\mathrm{(LHS)}
				&\leq
				\left(\frac{1}{\eta_x} - \mu_x - \beta_x\mu_{yx}^2\right)\sqn{x^{k} - x^*}
				+
				\left(\frac{1}{\eta_y} - \mu_y - \beta_y\mu_{xy}^2\right)\sqn{y^{k} - y^*}
				\\&-
				\frac{1}{4\eta_x}
				\E_{\xi_x^k, \xi_y^k}\sqn{x^{k+1} - x^k}
				-
				\frac{1}{4\eta_y}
				\E_{\xi_x^k, \xi_y^k}\sqn{y^{k+1} - y^k}
				\\&+		
				\left(\frac{2}{\sigma_x} - 1\right)\bg_f(x_f^k,x^*) 
				+
				\left(\frac{2}{\sigma_y} - 1\right)\bg_g(y_f^k,y^*) 
				\\&+
				2\E_{\xi_x^k, \xi_y^k}\<y^{k+1} -y^k,\mA(x^{k+1} - x^*)>
				-
				2\theta\<y^k -y^{k-1},\mA(x^{k} - x^*)>
				\\&+
				\frac{\theta}{4\eta_x}\E_{\xi_x^k, \xi_y^k}\sqn{x^{k+1} - x^k}
				+
				\frac{\theta}{4\eta_y}\sqn{y^k - y^{k-1}}
				\\&+ \frac{4}{1-\theta}\left(\delta_x(\xi^{k-1})+\delta_y(\xi^{k-1})\right) + \left(\beta_y + \frac{\omega}{2L_{xy}}\right)\sigma_f^2 + \left(\beta_x + \frac{1}{2L_{xy}\omega}\right)\sigma_g^2.
			\end{align*}
			Transforming further
			\begin{align*}
				\mathrm{(LHS)}
				&\leq
				\left(\frac{1}{\eta_x} - \mu_x - \beta_x\mu_{yx}^2\right)\sqn{x^{k} - x^*}
				+
				\left(\frac{1}{\eta_y} - \mu_y - \beta_y\mu_{xy}^2\right)\sqn{y^{k} - y^*}
				\\&+
				\frac{\theta}{4\eta_y}\sqn{y^k - y^{k-1}}
				-
				\frac{1}{4\eta_y}
				\E_{\xi_x^k, \xi_y^k}\sqn{y^{k+1} - y^k}
				\\&+		
				\left(\frac{2}{\sigma_x} - 1\right)\bg_f(x_f^k,x^*) 
				+
				\left(\frac{2}{\sigma_y} - 1\right)\bg_g(y_f^k,y^*) 
				\\&+
				2\E_{\xi_x^k, \xi_y^k}\<y^{k+1} -y^k,\mA(x^{k+1} - x^*)>
				-
				2\theta\<y^k -y^{k-1},\mA(x^{k} - x^*)>
				\\&+ \frac{4}{1-\theta}\left(\delta_x(\xi^{k-1})+\delta_y(\xi^{k-1})\right) + \left(\beta_y + \frac{\omega}{2L_{xy}}\right)\sigma_f^2 + \left(\beta_x + \frac{1}{2L_{xy}\omega}\right)\sigma_g^2.
			\end{align*}
			
			Using the definition of $\beta_x$ and $\beta_y$ we get
			\begin{align*}
				\mathrm{(LHS)}
				&\leq
				\left(1 - \eta_x\mu_x - \min\left\{\frac{\eta_x\mu_{yx}^2}{2L_y},\frac{\mu_{yx}^2}{2L_{xy}^2}\right\}\right)\frac{1}{\eta_x}\sqn{x^{k} - x^*}
				\\&+
				\left(1 - \eta_y\mu_y - \min\left\{\frac{\eta_y\mu_{xy}^2}{2L_x},\frac{\mu_{xy}^2}{2L_{xy}^2}\right\}\right)\frac{1}{\eta_y}\sqn{y^{k} - y^*}
				\\&+
				\frac{\theta}{4\eta_y}\sqn{y^k - y^{k-1}}
				-
				\frac{1}{4\eta_y}
				\E_{\xi_x^k, \xi_y^k}\sqn{y^{k+1} - y^k}
				\\&+		
				\left(\frac{2}{\sigma_x} - 1\right)\bg_f(x_f^k,x^*) 
				+
				\left(\frac{2}{\sigma_y} - 1\right)\bg_g(y_f^k,y^*) 
				\\&+
				2\E_{\xi_x^k, \xi_y^k}\<y^{k+1} -y^k,\mA(x^{k+1} - x^*)>
				-
				2\theta\<y^k -y^{k-1},\mA(x^{k} - x^*)>
				\\&+ \frac{4}{1-\theta}\left(\delta_x(\xi^{k-1})+\delta_y(\xi^{k-1})\right) + \frac{1}{2}\left(\frac{1}{L_x} + \frac{\omega}{L_{xy}}\right)\sigma_f^2 + \frac{1}{2}\left(\frac{1}{L_y} + \frac{1}{L_{xy}\omega}\right)\sigma_g^2
			\end{align*}
			Transforming this inequality
			\begin{align*}
				\mathrm{(LHS)}
				&\leq
				\left(1 - \max\left\{\eta_x\mu_x, \min\left\{\frac{\eta_x\mu_{yx}^2}{2L_y},\frac{\mu_{yx}^2}{2L_{xy}^2}\right\}\right\}\right)\frac{1}{\eta_x}\sqn{x^{k} - x^*}
				\\&+
				\left(1 - \max\left\{\eta_y\mu_y, \min\left\{\frac{\eta_y\mu_{xy}^2}{2L_x},\frac{\mu_{xy}^2}{2L_{xy}^2}\right\}\right\}\right)\frac{1}{\eta_y}\sqn{y^{k} - y^*}
				\\&+
				\frac{\theta}{4\eta_y}\sqn{y^k - y^{k-1}}
				-
				\frac{1}{4\eta_y}
				\E_{\xi_x^k, \xi_y^k}\sqn{y^{k+1} - y^k}
				\\&+		
				\left(\frac{2}{\sigma_x} - 1\right)\bg_f(x_f^k,x^*) 
				+
				\left(\frac{2}{\sigma_y} - 1\right)\bg_g(y_f^k,y^*) 
				\\&+
				2\E_{\xi_x^k, \xi_y^k}\<y^{k+1} -y^k,\mA(x^{k+1} - x^*)>
				-
				2\theta\<y^k -y^{k-1},\mA(x^{k} - x^*)>
				\\&+ \frac{4}{1-\theta}\left(\delta_x(\xi^{k-1})+\delta_y(\xi^{k-1})\right) + \frac{1}{2}\left(\frac{1}{L_x} + \frac{\omega}{L_{xy}}\right)\sigma_f^2 + \frac{1}{2}\left(\frac{1}{L_y} + \frac{1}{L_{xy}\omega}\right)\sigma_g^2.
			\end{align*}
			Using the definition of $\theta$ we get
			\begin{align*}
				\mathrm{(LHS)}
				&\leq
				\theta\left(
				\frac{1}{\eta_x}\sqn{x^{k} - x^*}
				+
				\frac{1}{\eta_y}\sqn{y^{k} - y^*}
				+
				\frac{1}{4\eta_y}\sqn{y^k - y^{k-1}}
				\right)
				\\&+
				\theta\left(
				-2\<y^k -y^{k-1},\mA(x^{k} - x^*)>
				+
				\frac{2}{\sigma_x}\bg_f(x_f^k,x^*) 
				+
				\frac{2}{\sigma_y}\bg_g(y_f^k,y^*) 
				\right)	
				\\&-
				\frac{1}{4\eta_y}
				\E_{\xi_x^k, \xi_y^k}\sqn{y^{k+1} - y^k}
				+
				2\E_{\xi_x^k, \xi_y^k}\<y^{k+1} -y^k,\mA(x^{k+1} - x^*)>
				\\&+ \frac{4}{1-\theta}\left(\delta_x(\xi^{k-1})+\delta_y(\xi^{k-1})\right) + \frac{1}{2}\left(\frac{1}{L_x} + \frac{\omega}{L_{xy}}\right)\sigma_f^2 + \frac{1}{2}\left(\frac{1}{L_y} + \frac{1}{L_{xy}\omega}\right)\sigma_g^2.
			\end{align*}

			After taking the expectation over all random variables, rearranging and using the definition of $\Psi^k$, using the fact that $\E\delta_x(\xi^{k-1}) \leq \delta_x, \E\delta_y(\xi^{k-1}) \leq \delta_y$ we get
			\begin{equation*}
				\E\Psi^{k+1}\leq \theta\E\Psi^k + \frac{4}{1-\theta}\left(\delta_x+\delta_y\right) + \frac{1}{2}\left(\frac{1}{L_x} + \frac{\omega}{L_{xy}}\right)\sigma_f^2 + \frac{1}{2}\left(\frac{1}{L_y} + \frac{1}{L_{xy}\omega}\right)\sigma_g^2.
			\end{equation*}
			Finally, using the definition of $\Psi^k$, $\eta_x$ and $\eta_y$ we get
			\begin{align*}
				\Psi^k
				&\geq
				\frac{1}{\eta_x}\sqn{x^{k} - x^*}
				+
				\frac{1}{\eta_y}\sqn{y^{k} - y^*}
				+
				\frac{1}{4\eta_y}\sqn{y^k - y^{k-1}}
				-
				2\<y^k -y^{k-1},\mA(x^{k} - x^*)>
				\\&\geq
				\frac{1}{\eta_x}\sqn{x^{k} - x^*}
				+
				\frac{1}{\eta_y}\sqn{y^{k} - y^*}
				+
				\frac{1}{4\eta_y}\sqn{y^k - y^{k-1}}
				-
				2L_{xy}\norm{y^k -y^{k-1}}\norm{x^{k} - x^*}
				\\&\geq
				\frac{1}{\eta_x}\sqn{x^{k} - x^*}
				+
				\frac{1}{\eta_y}\sqn{y^{k} - y^*}
				+
				\frac{1}{4\eta_y}\sqn{y^k - y^{k-1}}
				-
				\frac{1}{2\sqrt{\eta_x\eta_y}}\norm{y^k -y^{k-1}}\norm{x^{k} - x^*}
				\\&\geq
				\frac{3}{4\eta_x}\sqn{x^{k} - x^*}
				+
				\frac{1}{\eta_y}\sqn{y^{k} - y^*}.
			\end{align*}
		\end{proof}
		
		Back to proof of the Theorem \ref{th4}.
		
		Let $\Sigma^2 \triangleq \left( \frac{1}{L_x}+\frac{\omega}{L_{xy}} \right)\sigma_f^2+\left( \frac{1}{L_y}+\frac{1}{L_{xy}\omega} \right)\sigma_g^2$. Then
		
		\begin{align*}
			\E\Psi^k &\leq \theta^k \Psi^0 + \left(\frac{4}{1-\theta}(\delta_x+\delta_y)+\frac{\Sigma^2}{2}\right)(1 + \theta + \theta^2 + \dots)
			\\&\leq \theta^k \Psi^0 +  \frac{4}{(1-\theta)^2}(\delta_x+\delta_y)+\frac{\Sigma^2}{2(1-\theta)},
		\end{align*}
		
		$$\theta^k \Psi^0 +  \frac{4}{(1-\theta)^2}(\delta_x+\delta_y)+\frac{\Sigma^2}{2(1-\theta)}
		\geq
		\E \Psi^k
		\geq \frac{3}{4\eta_x}\E\|x^k - x^*\|^2 + \frac{1}{\eta_y}\E\|y^k - y^*\|^2.$$
		
		Using the definitions of $\eta_x$ and $\eta_y$, we get
		
		\begin{equation}\label{syn_th4_x}
			\mathbb{E}\|x^k - x^*\|^2 \leq \frac{\omega}{3L_{xy}}\left( \theta^k \Psi^0 +  \frac{4}{(1-\theta)^2}(\delta_x+\delta_y)+\frac{\Sigma^2}{2(1-\theta)}\right),
		\end{equation}
		\begin{equation}\label{syn_th4_y}
			\mathbb{E}\|y^k - y^*\|^2 \leq \frac1{4L_{xy}\omega}\left( \theta^k \Psi^0 +  \frac{4}{(1-\theta)^2}(\delta_x+\delta_y)+\frac{\Sigma^2}{2(1-\theta)}\right).
		\end{equation}
		
		Also for such definitions we know from \cite{minimax2112}
		{\small
		\begin{align*}\label{ests}
			\frac{1}{\rho_a} &\leq
			4 + 4\max\left\{\sqrt{\frac{L_x}{\mu_x}}, \sqrt{\frac{L_y}{\mu_y}},\frac{L_{xy}}{\sqrt{\mu_x\mu_y}}\right\} \\
			&\quad\text{ for } \omega = \sqrt{\frac{\mu_y}{\mu_x}}, \sigma_x = \sqrt{\frac{\mu_x}{2L_x}},\sigma_y = \sqrt{\frac{\mu_y}{2L_y}},\\
			\frac{1}{\rho_b} & \leq
			4+8\max\left\{
			\frac{\sqrt{L_xL_y}}{\mu_{xy}},
			\frac{L_{xy}}{\mu_{xy}}\sqrt{\frac{L_x}{\mu_x}},
			\frac{L_{xy}^2}{\mu_{xy}^2}
			\right\} \\
			&\quad\text{ for } \omega = \sqrt{\frac{\mu_{xy}^2}{2\mu_xL_x}}, \sigma_x = \sqrt{\frac{\mu_x}{2L_x}},\sigma_y =\min\left\{1,\sqrt{\frac{\mu_{xy}^2}{4L_xL_y}}\right\},\\
			\frac{1}{\rho_c} & \leq
			4+8\max\left\{
			\frac{\sqrt{L_xL_y}}{\mu_{yx}},
			\frac{L_{xy}}{\mu_{yx}}\sqrt{\frac{L_y}{\mu_y}},
			\frac{L_{xy}^2}{\mu_{yx}^2}
			\right\} \\
			&\quad\text{ for } \omega = \sqrt{\frac{2\mu_yL_y}{\mu_{yx}^2}},\sigma_x =\min\left\{1,\sqrt{\frac{\mu_{yx}^2}{4L_xL_y}}\right\},\sigma_y = \sqrt{\frac{\mu_y}{2L_y}},\\
			\frac{1}{\rho_d} &\leq 2+8\max\left\{
			\frac{\sqrt{L_xL_y}L_{xy}}{\mu_{xy}\mu_{yx}},
			\frac{L_{xy}^2}{\mu_{yx}^2},
			\frac{L_{xy}^2}{\mu_{xy}^2}
			\right\} \\
			&\quad\text{ for } \omega = \frac{\mu_{xy}}{\mu_{yx}}\sqrt{\frac{L_y}{L_x}}, \sigma_x = \min\left\{1,\sqrt{\frac{\mu_{yx}^2}{4L_xL_y}}\right\},\sigma_y =\min\left\{1,\sqrt{\frac{\mu_{xy}^2}{4L_xL_y}}\right\}, \\
            \frac{1}{1-\theta} &= \min \{ \rho_a^{-1},  \rho_b^{-1},  \rho_c^{-1},  \rho_d^{-1}\}.
		\end{align*}
        }
		
		% \begin{align*}
		% 	\frac{1}{1-\theta} &= \min \{ \rho_a^{-1},  \rho_b^{-1},  \rho_c^{-1},  \rho_d^{-1}\}.
		% \end{align*}
		
		Note, that  adding up batches and choosing $\omega = \sqrt{\frac{\mu_y}{\mu_x}}, \sigma_x = \sqrt{\frac{\mu_x}{2L_x}},\sigma_y = \sqrt{\frac{\mu_x}{2L_x}}$ proves the Theorem \ref{th4}.
		
		Rewriting inequalities in batch setting and assuming $\delta_x=\delta_y=0$ we get
		
		$$\mathbb{E}\|x^k - x^*\|^2 \leq \frac{\omega}{3L_{xy}}\left( \theta^k \Psi^0+\frac{1}{2(1-\theta)}\left(\left( \frac{1}{L_x}+\frac{\omega}{L_{xy}}\right)\frac{\sigma_f^2}{r_f}+\left( \frac{1}{L_y}+\frac{1}{L_{xy}\omega} \right)\frac{\sigma_g^2}{r_g}\right) \right),$$
		$$\mathbb{E}\|y^k - y^*\|^2 \leq \frac{1}{4L_{xy}\omega}\left( \theta^k \Psi^0+\frac{1}{2(1-\theta)}\left(\left( \frac{1}{L_x}+\frac{\omega}{L_{xy}}\right)\frac{\sigma_f^2}{r_f}+\left( \frac{1}{L_y}+\frac{1}{L_{xy}\omega}\right)\frac{\sigma_g^2}{r_g}\right) \right).$$
		
		Therefore, we can estimate the number of algorithm iterations $N = \mathcal{O} \left(\frac{1}{1-\theta}\log{\frac{C}{\epsilon}}\right)$, where C is polynomial and not depend on $\epsilon$. Rewriting it we obtain $N = \mathcal{O} \left( \min \{ \rho_a^{-1},  \rho_b^{-1},  \rho_c^{-1},  \rho_d^{-1}\}\log{\frac{C}{\epsilon}}\right).$
		
		It is sufficient to take such batch sizes $r_f = \left\lceil\frac{\max\{\omega, \omega^{-1}\}}{2L_{xy}(1-\theta)\epsilon}\left(\frac1{L_x}+\frac{\omega}{L_{xy}}\right)\sigma_f^2\right\rceil$, $r_g = \left\lceil\frac{\max\{\omega, \omega^{-1}\}}{2L_{xy}(1-\theta)\epsilon}\left(\frac1{L_y}+\frac{1}{L_{xy}\omega}\right)\sigma_g^2\right\rceil.$
		
		Rewriting it with the selected constants
		$$r_f = \left\lceil \max\left\{\sqrt{\frac{L_x}{\mu_x}}, \sqrt{\frac{L_y}{\mu_y}},\frac{L_{xy}}{\sqrt{\mu_x\mu_y}}\right\} \frac{\mu}{2L_{xy}\sqrt{\mu_x\mu_y}\epsilon}\left(\frac{1}{L_x} + \frac{1}{L_{xy}}\sqrt{\frac{\mu_y}{\mu_x}}\right)\sigma_f^2 \right\rceil,$$
		$$r_g = \left\lceil\max\left\{\sqrt{\frac{L_x}{\mu_x}}, \sqrt{\frac{L_y}{\mu_y}},\frac{L_{xy}}{\sqrt{\mu_x\mu_y}}\right\} \frac{\mu}{2L_{xy}\sqrt{\mu_x\mu_y}\epsilon}\left(\frac{1}{L_y} + \frac{1}{L_{xy}}\sqrt{\frac{\mu_x}{\mu_y}}\right)\sigma_g^2\right\rceil,$$
		
		where $\mu = \max\{\mu_x, \mu_y\}$.
		
		\section{Decentralized setting}\label{app:consensus_iterations}
		
		Let us get an Algorithm \ref{alg:AV} from the Algorithm \ref{alg:DAPDG}, by multiplying every line by $\frac{\textbf{1}}{n}$, where \textbf{1} is a column of $1$.
	
	Algorithm \ref{alg:AV} shows what happens to the average values at the nodes.
	
	\begin{algorithm}[H]\caption{Average values}\label{alg:AV}
		\begin{algorithmic}
			\State \textbf{Input:} {$\eta_x, \eta_y, \alpha_x, \alpha_y, \beta_x, \beta_y > 0$, $\tau_x, \tau_y, \sigma_x, \sigma_y\in \left(0, 1\right]$, $\theta\in \left(0, 1\right)$
				
				\State $\overline{x}^0_f = \overline{x}^0$, $\overline{x}^0 \in \range \mA^\top$
				
				\State $\overline{y}^0_f = \overline{y}^{-1} = \overline{y}^0$, $\overline{y}^0 \in \range \mA$}
			
			\For{$k=0,1,2,\dots$}
			\State $\overline{y}_m^k = \overline{y}^k + \theta\left(\overline{y}^k-\overline{y}^{k-1}\right)$\label{AV:1}
			
			\State $\overline{x}_g^k = \tau_x \overline{x}^k + \left(1 - \tau_x\right)\overline{x}_f^k$\label{AV:2}
			
			\State $\overline{y}_g^k = \tau_y \overline{y}^k + \left(1 - \tau_y\right)\overline{y}_f^k$\label{AV:3}
			
			\State $\overline{x}^{k+1}=\overline{x}^k+\eta_x\alpha_x\left(\overline{x}_g^k - \overline{x}^k\right) - \eta_x\beta_x A^T\left(A \overline{x}^k-\nabla g_{\delta}(\overline{y}_g^k, \xi_y^k)\right) - \eta_x\left(\nabla f_{\delta}(\overline{x}_g^k, \xi_x^k) + A^T \overline{y}_m^k\right)$\label{AV:4}
			
			\State $\overline{y}^{k+1} = \overline{y}^k + \eta_y\alpha_y\left(\overline{y}_g^k - \overline{y}^k\right) - \eta_y\beta_y A\left(A^T \overline{y}^k+ f_{\delta}(\overline{x}_g^k, \xi_x^k)\right) - \eta_y\left(\nabla g_{\delta}(\overline{y}_g^k, \xi_y^k) - A \overline{x}^{k+1}\right)$\label{AV:5}
			
			\State $\overline{x}_f^{k+1} = \overline{x}_g^k + \sigma_x\left(\overline{x}^{k+1} - \overline{x}^k\right)$\label{AV:6}
			
			\State $\overline{y}_f^{k+1} = \overline{y}_g^k + \sigma_y\left(\overline{y}^{k+1} - \overline{y}^k\right)$\label{AV:7}
			
			\EndFor
		\end{algorithmic}
	\end{algorithm}
	
	Supporting values $X$ and $Y$ to be in the neighborhood of $\mathcal{C}(d_x)$ and $\mathcal{C}(d_y)$ and using Lemma~\ref{lemma:3}, conditions of  Theorem~\ref{th4} hold.
	
	Using Assumption~\ref{ass:3}, Assumption~\ref{ass:4}, we get
	\begin{equation}\label{c:10}
		\begin{split}
			\mathbb{E}_{\xi_{x, k}}\|f_{\delta}(\overline{x}_g^k, \xi_{x, k}) -  \nabla f_{\delta}(\overline{x}_g^k)\|^2 \leq \frac{\sum_{i=1}^n \sigma_{f, i}^2/r_{f, i}}{n^2} \triangleq \frac{\sigma_{F, r}^2}{n},
			\\
			\mathbb{E}_{\xi_{y, k}}\|g_{\delta}(\overline{y}_g^k, \xi_{y, k}) -  \nabla g_{\delta}(\overline{y}_g^k)\|^2 \leq \frac{\sum_{i=1}^n \sigma_{g, i}^2/r_{g, i}}{n^2} \triangleq \frac{\sigma_{G,r}^2}{n}.
		\end{split}
	\end{equation}
		
		Let us support the number of iterations of Consensus to be sufficiently big to guarantee $\mathbb{E}\|X^k - \overline{X^k}\| \leq \sqrt{\delta'}$ and $\mathbb{E}\|Y^k - \overline{Y^k}\| \leq \sqrt{\delta'}.$
		
		Introducing some definitions, which correspond to Lemma \ref{lemma:3}
		$$\delta_x = \frac{1}{2n}\left( \frac{L_{lx}^2}{L_{x}} + \frac{2L_{lx}^2}{\mu_{x}} + L_{lx} - \mu_{lx} \right)\delta',$$
		$$\delta_y = \frac{1}{2n}\left( \frac{L_{ly}^2}{L_{y}} + \frac{2L_{ly}^2}{\mu_{y}} + L_{ly} - \mu_{ly} \right)\delta',$$
		
		$\hat{L_x} = 2L_{x}$, $\hat{L_y} = 2L_{y}$, $\hat{\mu_x} = \mu_{x}/2$, $\hat{\mu_y} = \mu_{y}/2.$
		
		Consider the iteration $k \geq 1$. Assuming, that $\mathbb{E}\|X^t - \overline{X^t}\| \leq \sqrt{\delta'}$ and $\mathbb{E}\|Y^t - \overline{Y^t}\| \leq \sqrt{\delta'}$ for $t=0,1,\dots,k$, we are going to prove it for $t=k+1$, using constant number of consensus iterations.

		Using Line~\ref{DAPDG:line:6} and \ref{DAPDG:line:2} of Algorithm~\ref{alg:DAPDG}, we get
		$$X_g^k = \tau_x X^k + (1-\tau_x) X_f^k = (\tau_x+(1-\tau_x)\sigma_x) X^k - (1-\tau_x)\sigma_x X^{k-1} + (1-\tau_x)X_g^{k-1}.$$
		
		Define $V^k = X_g^k-\sigma_x X^k$. Using $X_g^0 = X^0$, we get $V^0 = (1-\sigma_x) X^0, \|V^0-\overline{V^0}\|\leq(1-\sigma_x)\sqrt{\delta'}.$
		
		\begin{align*}
			V^k &= (1-\sigma_x)\tau_x X^k + (1-\tau_x)V^{k-1},
			\\
			V^k-\overline{V^k} &= (1-\sigma_x)\tau_x \left( X^k-\overline{X^k} \right) + (1-\tau_x)\left( V^{k-1}-\overline{V^{k-1}} \right),
			\\
			\E\|V^k-\overline{V^k}\|&\leq (1-\sigma_x)\tau_x\sqrt{\delta'} + (1-\tau_x)(1-\sigma_x)\sqrt{\delta'} = (1-\sigma_x)\sqrt{\delta'}.
		\end{align*}
		
		Let us now estimate $X_f^k$, $k\geq 1$. Using Line~\ref{DAPDG:line:6}, we get
		$$X_f^k = V^{k-1} + \sigma_x X^k,$$
		$$\mathbb{E}\|X_f^k - \overline{X_f^k}\| \leq \sqrt{\delta'}.$$
		
		Let us now estimate $X_g^k$ and $Y_m^k$. Using Line~\ref{DAPDG:line:2} and \ref{DAPDG:line:1}, we get
		$$\mathbb{E}\|X_g^k - \overline{X_g^k}\| \leq \sqrt{\delta'},$$
		$$\mathbb{E}\|Y_m^k - \overline{Y_m^k}\| \leq (1+2\theta)\sqrt{\delta'}.$$
		
		The estimations for $Y_g^k$, $Y_f^k$ are similar.
		
		Let us now estimate $\mathbb{E}\|U^{k+1}-\overline{U^{k+1}}\|$. Using Line~\ref{DAPDG:line:4}, we get
		
		\begin{align*}
			U^{k+1}-\overline{U^{k+1}}&=(1-\eta_x\alpha_x)\left(X^k-\overline{X^k}\right)+\eta_x\alpha_x\left( X_g^k-\overline{X_g^k}\right)
			\\&- \eta_x\beta_x A^T\left(A \left(X^k-\overline{X^k}\right)-\left(\nabla^r G(Y_g^k, \xi_{y, k})-\overline{\nabla^r G(Y_g^k, \xi_{y, k})}\right)\right)
			\\&- \eta_x\left(\left(\nabla^r F(X_g^k, \xi_{x, k})-\overline{\nabla^r F(X_g^k, \xi_{x, k}})\right) + A^T \left(Y_m^k-\overline{Y_m^k}\right)\right).
		\end{align*}
		
		Using that $\eta_x\alpha_x \leq 1$ and previous estimations, we get
		
		\begin{align*}
			\mathbb{E}\|U^{k+1}-\overline{U^{k+1}}\| &\leq (1-\eta_x\alpha_x)\sqrt{\delta'} + \eta_x\alpha_x\sqrt{\delta'} + \eta_x\beta_x L_{xy}^2\sqrt{\delta'}
			\\&+
			\eta_x\beta_x L_{xy}\E\|\nabla^r G(Y_g^k, \xi_{y, k})\|+\eta_x\mathbb{E}\|\nabla^r F(X_g^k, \xi_{x, k})\|+\eta_x L_{xy}(1+2\theta)\sqrt{\delta'}
			\\&= (1+\eta_x\beta_x L_{xy}^2+\eta_x L_{xy}(1+2\theta))\sqrt{\delta'}+\eta_x\beta_x L_{xy}\mathbb{E}\|\nabla^r G(Y_g^k, \xi_{y, k})\|
			\\&+
			\eta_x\mathbb{E}\|\nabla^r F(X_g^k, \xi_{x, k})\|.
		\end{align*}

		Getting estimations for $\E\norm{\nabla^r G(Y_g^k, \xi_{y, k})}$ and $\E\norm{\nabla^r F(X_g^k, \xi_{x, k})}$.
		
		\begin{align*}
			\E\norm{\nabla^r F(X_g^k, \xi_{x, k})} &\leq \E\norm{\nabla^r F(X_g^k, \xi_{x, k}) - \nabla F(X_g^k)} + \E\norm{\nabla F(X_g^k) - \nabla F(\ovl{X_g^k})}
			\\&+
			\E\|\nabla F(\overline{X_g^k}) - \nabla F(X^*)\| + \|\nabla F(X^*)\| 
			\\&\leq
			\left(\E\sqn{\nabla^r F(X_g^k, \xi_{x, k}) - \nabla F(X_g^k)}\right)^{\frac{1}{2}} + L_{lx}\E\norm{X_g^k-\ovl{X_g^k}}
			\\&+
			L_{x}\E\norm{\ovl{X_g^k}-X^*} + \norm{\nabla F(X^*)}
			\\&
			\leq \left(\sum_{i=1}^n\sigma_{f, i}^2/r_{f, i}\right)^{\frac{1}{2}} + L_{lx}\sqrt{\delta'} + L_{x}\sqrt{n}\E\norm{\ovl{x_g^k}-x^*} + \norm{\nabla F(X^*)}.
		\end{align*}
		
		Let us define $M_x$
		$$M_x^2=\frac{\omega}{3L_{xy}}\left( \Psi^0 +  \frac{4}{(1-\theta)^2}(\delta_x+\delta_y)+\frac{\Sigma^2}{2(1-\theta)} \right),$$
		$$\Sigma^2 = \left( \frac{1}{2L_{x}}+\frac{\omega}{L_{xy}} \right)\frac{\sigma_{F, r}^2}{n}+\left( \frac{1}{2L_{y}}+\frac{1}{L_{xy}\omega} \right)\frac{\sigma_{G, r}^2}{n}.$$
		
		We choose constants the same as in Equation~\ref{syn_th4_x} and Equation~\ref{syn_th4_y} for Algorithm \ref{alg:AV}.
		
		Now we are going to estimate $\norm{\ovl{x_g^k}-x^*}$. As we know from Equation~\ref{syn_th4_x} and from Equation \ref{c:10}
		
		$$\E\sqn{x^k - x^*}  \leq M_x^2,$$
		
		$$\E\norm{x^k - x^*} \leq \sqrt{\E\sqn{x^k - x^*}} \leq M_x.$$
		
		Let $k \geq 1$. Using Line~\ref{AV:6} and \ref{AV:2} of Algorithm \ref{alg:AV}, we get
		$$
		\overline{x_g^k} = \tau_x\overline{x^k} + (1-\tau_x)\overline{x_f^k} = (\tau_x+(1-\tau_x)\sigma_x) \overline{x^k} - (1-\tau_x)\sigma_x \overline{x^{k-1}} + (1-\tau_x)\overline{x_g^{k-1}}.
		$$
		
		Let's define $\overline{v^k} = \overline{x_g^k} - \sigma_x\overline{x^k}$ and $v^* = (1-\sigma_x)x^*$. $\overline{v^0} = (1-\sigma_x)\overline{x^0}$, therefore $\|\overline{v^0}-v^*\| \leq (1-\sigma_x)M_x$.
		
		$$\overline{v^k} = \tau_x(1-\sigma_x)\overline{x^k}+(1-\tau_x)\overline{v^{k-1}}.$$
		
		Firstly, we want to estimate $\E\|\overline{v^k}-v^*\|$.
		
		\begin{align*}
			\E\|\overline{v^k}-v^*\| &\leq \tau_x(1-\sigma_x)\E\|\overline{x^k}-x^*\|+(1-\tau_x)\E\|\overline{v^{k-1}}-v^*\|
			\\&\leq
			(\tau_x(1-\sigma_x)+(1-\tau_x)(1-\sigma_x))M_x=(1-\sigma_x)M_x.
		\end{align*}
		
		Using Line~\ref{AV:6}, we get
		
		$$\overline{x_f^k} = \overline{v^{k-1}}+\sigma_x\overline{x^k},$$
		$$\E\|\overline{x_f^k}-x^*\| \leq \E\|\overline{v^{k-1}}-v^*\|+\sigma_x\E\|\overline{x^k}-x^*\|\leq (1-\sigma_x)M_x+\sigma_x M_x = M_x.$$
		
		Let's estimate $\E\|\overline{x_g^k}-x^*\|$. Using Line~\ref{AV:1}, we get
		
		$$\E\|\overline{x_g^k}-x^*\| \leq \tau_x\E\|\overline{x^k}-x^*\|+(1-\tau_x)\E\|\overline{x_f^k}-x^*\| \leq M_x.$$
		
		Returning to $\E\norm{\nabla^r F(X_g^k, \xi_{x, k})}$
		$$\E\norm{\nabla^r F(X_g^k, \xi_{x, k})} \leq \sqrt{n\sigma_{F, r}^2} + L_{lx}\sqrt{\delta'} + L_{x}\sqrt{n}M_x + \norm{\nabla F(X^*)}.$$
		
		Let's define $M_y$
		$$M_y^2=\frac{1}{4L_{xy}\omega}\left( \Psi^0 +  \frac{4}{(1-\theta)^2}(\delta_x+\delta_y)+\frac{\Sigma^2}{2(1-\theta)} \right).$$
		
		Then we can estimate $\E\norm{\nabla^r G(Y_g^k, \xi_{y, k})}$ in a similar way
		
		$$\E\norm{\nabla^r G(Y_g^k, \xi_{y, k})} \leq \sqrt{n\sigma_{G, r}^2} + L_{ly}\sqrt{\delta'} + L_{y}\sqrt{n}M_y + \norm{\nabla G(Y^*)}.$$
		
		\begin{lemma}\label{l5}
			$$\max\left\{ \E\norm{U^{k+1}-\ovl{U^{k+1}}}, \E\norm{W^{k+1}-\ovl{W^{k+1}}} \right\} \leq D,$$
			
			where
			\begin{equation}
				D = \max\left\{ D_{x,1}\sqrt{\delta'}+D_{x,2}, D_{y,1}\sqrt{\delta'}+D_{y,2} \right\},
			\end{equation}
			
			\begin{equation}
				\begin{split}
					D_{y, 2} &=\frac{L_{xy}}{2\mu_{y}}D_{x,2}+ \frac{1}{2L_{xy}}\left(\sqrt{n\sigma_{F,r}^2} + L_{x}\sqrt{n}M_x+\|\nabla F(X^*)\| \right)
					\\&+
					\frac{1}{2\mu_{y}}\left(\sqrt{n\sigma_{G,r}^2} + L_{y}\sqrt{n}M_y+\|\nabla G(Y^*)\| \right),
				\end{split}
			\end{equation}
			
			\begin{equation}
				D_{y,1} = \frac{3}{2}+\frac{L_{xy}}{2\mu_{y}}D_{x,1}+\frac{L_{lx}}{2L_{xy}}+\frac{L_{ly}}{2\mu_{y}},
			\end{equation}
			
			\begin{equation}
				\begin{split}
					D_{x,2} &= \frac{1}{2L_{xy}}\left(\sqrt{n\sigma_{G,r}^2} + L_{y}\sqrt{n}M_y+\|\nabla G(Y^*)\| \right)
					\\&+
					\frac{1}{2\mu_{x}}\left(\sqrt{n\sigma_{F,r}^2} + L_{x}\sqrt{n}M_x+\|\nabla F(X^*)\| \right),
				\end{split}
			\end{equation}
			
			\begin{equation}
				D_{x,1} = \frac{3}{2}+\frac{L_{xy}}{2\mu_{x}}(1+2\theta)+\frac{L_{ly}}{2L_{xy}}+\frac{L_{lx}}{2\mu_{x}},
			\end{equation}
			
			\begin{equation}
				\begin{split}
					M_x^2&=\frac{\omega}{3L_{xy}}\left(\Psi^0 +  \frac{4}{(1-\theta)^2}(\delta_x+\delta_y)+\frac{\Sigma^2}{2(1-\theta)}\right),
					\\
					M_y^2&=\frac{1}{4L_{xy}\omega}\left(\Psi^0 +  \frac{4}{(1-\theta)^2}(\delta_x+\delta_y)+\frac{\Sigma^2}{2(1-\theta)}\right),
				\end{split}
			\end{equation}
			
			\begin{equation}
				\Sigma^2 = \left(\frac{1}{2L_{x}}+\frac{\omega}{L_{xy}} \right)\frac{\sigma_{F, r}^2}{n}
				+
				\left(\frac{1}{2L_{y}}+\frac{1}{L_{xy}\omega}\right)\frac{\sigma_{G, r}^2}{n}.
			\end{equation}
			
		\end{lemma}
		\begin{proof}
			\begin{align*}
				\E\|U^{k+1}-\overline{U^{k+1}}\| &\leq (1+\eta_x\beta_x L_{xy}^2+\eta_x L_{xy}(1+2\theta))\sqrt{\delta'}
				\\&+
				\eta_x\beta_x L_{xy}\E\|\nabla^r G(Y_g^k, \xi_{y, k})\|+\eta_x\E\|\nabla^r F(X_g^k, \xi_{x, k})\|.
			\end{align*}
			Using the definition of $\eta_x$, $\beta_x$ and estimations on gradients, we get
			\begin{align*}
				\begin{split}
					\E\|U^{k+1}-\overline{U^{k+1}}\| &\leq \left( 1+\frac{1}{2}+\frac{L_{xy}}{4\hat{\mu_x}}(1+2\theta)\right)\sqrt{\delta'}+\frac{1}{2L_{xy}}\E\|\nabla^r G(Y_g^k, \xi_{y, k})\|
					\\&+
					\frac{1}{4\hat{\mu_x}}\E\|\nabla^r F(X_g^k, \xi_{x, k})\|
					\\&\leq
					\left(\frac{3}{2}+\frac{L_{xy}}{4\hat{\mu_x}}(1+2\theta)+\frac{L_{ly}}{2L_{xy}}+\frac{L_{lx}}{4\hat{\mu_x}}\right)\sqrt{\delta'} 
					\\&+
					\frac{1}{2L_{xy}}\left(\sqrt{n\sigma_{G, r}^2} + L_{y}\sqrt{n}M_y+\|\nabla G(Y^*)\| \right)
					\\&+
					\frac{1}{4\hat{\mu_x}}\left(\sqrt{n\sigma_{F, r}^2} + L_{x}\sqrt{n}M_x+\|\nabla F(X^*)\| \right) = D_{x,1}\sqrt{\delta'}+D_{x,2}.
				\end{split}
			\end{align*}
			
			Let's estimate $\E\|W^{k+1}-\overline{W^{k+1}}\|$
			\begin{align*}
				W^{k+1}-\overline{W^{k+1}}&=(1-\eta_y\alpha_y)\left(Y^k-\overline{Y^k}\right)+\eta_y\alpha_y\left( Y_g^k-\overline{Y_g^k}\right)
				\\&- 
				\eta_y\beta_y A\left(A^T \left(Y^k-\overline{Y^k}\right)+\left(\nabla^r F(X_g^k, \xi_{x, k})-\ovl{\nabla^r F(X_g^k, \xi_{x, k})}\right)\right) 
				\\&- \eta_y\left(\left(\nabla^r G(Y_g^k, \xi_{y, k})-\ovl{\nabla^r G(Y_g^k, \xi_{y, k})}\right) - A \left(U^{k+1}-\overline{U^{k+1}}\right)\right).
			\end{align*}
			
			Using that $\eta_y\alpha_y \leq 1$ and previous estimations, we get
			
			\begin{align*}
				\E\|W^{k+1}-\overline{W^{k+1}}\| &\leq (1-\eta_y\alpha_y)\sqrt{\delta'} + \eta_y\alpha_y\sqrt{\delta'} + \eta_y\beta_y L_{xy}^2\sqrt{\delta'} 
				\\&+
				\eta_y\beta_y L_{xy}\E\|\nabla^r F(X_g^k, \xi_{x, k})\|
				\\&+
				\eta_y\E\|\nabla^r G(Y_g^k, \xi_{y, k})\|+\eta_y L_{xy}\E\|U^{k+1}-\overline{U^{k+1}}\| 
				\\&
				= (1+\eta_y\beta_y L_{xy}^2)\sqrt{\delta'}+\eta_y L_{xy}\E\|U^{k+1}-\overline{U^{k+1}}\|
				\\&+
				\eta_y\beta_y L_{xy}\E\|\nabla^r F(X_g^k, \xi_{x, k})\|
				+
				\eta_y\E\|\nabla^r G(Y_g^k, \xi_{y, k})\|.
			\end{align*}
			Using the definition of $\beta_y$, $\eta_y$ and gradient estimations, we get
			\begin{align*}
				\E\|W^{k+1}-\overline{W^{k+1}}\| &\leq \left(1+\
				\frac{1}{2}\right)\sqrt{\delta'}+\frac{L_{xy}}{4\hat{\mu_y}}\E\|U^{k+1}-\overline{U^{k+1}}\|+\frac{1}{2L_{xy}}\E\|\nabla^r F(X_g^k, \xi_{x, k})\|
				\\&+
				\frac{1}{4\hat{\mu_y}}\E\|\nabla^r G(Y_g^k, \xi_{y, k})\|
				\\&\leq 
				\left(\frac{3}{2}+\frac{L_{lx}}{2L_{xy}}+\frac{L_{ly}}{4\hat{\mu_y}}+\frac{L_{xy}}{4\hat{\mu_y}}D_{x,1} \right)\sqrt{\delta'}+\frac{L_{xy}}{4\hat{\mu_y}}D_{x,2}
				\\&+
				\frac{1}{2L_{xy}}\left(\sqrt{n\sigma_{F, r}^2} + L_{x}\sqrt{n}M_x+\|\nabla F(X^*)\| \right)
				\\&+
				\frac{1}{4\hat{\mu_y}}\left(\sqrt{n\sigma_{G, r}^2} + L_{y}\sqrt{n}M_y+\|\nabla G(Y^*)\| \right) = D_{y,1}\sqrt{\delta'}+D_{y,2}.
			\end{align*}
			
		\end{proof}
		
		Now let us estimate the number of communication steps $T$.
		
		$$(1-\lambda)^{T/\tau}\max\{\E\|W^{k+1}-\overline{W^{k+1}}\|, \E\|U^{k+1}-\overline{U^{k+1}}\| \}  \leq \delta'.$$
		
		It would we sufficient to guarantee
		$$(1-\lambda)^{T/\tau}D  \leq \delta'.$$
		
		Above inequality leads from this
		$$T \geq \frac{\tau}{\lambda}\log\left(\frac{D}{\delta'}\right).$$
		
		\paragraph{Putting the proof together}
		
		Using Lemma \ref{l1}, we get
		$$\E\sqn{\ovl{x^k} - x^*} \leq \frac{\omega}{3L_{xy}}\left( \theta^k\Psi^0 +  \frac{4}{(1-\theta)^2}(\delta_x+\delta_y)+
		\frac{\Sigma^2}{2(1-\theta)} \right),$$
		$$\E\sqn{\ovl{y^k} - y^*} \leq \frac{1}{4L_{xy}\omega}\left( \theta^k\Psi^0 + 
		\frac{4}{(1-\theta)^2}(\delta_x+\delta_y)+\frac{\Sigma^2}{2(1-\theta)} \right).$$
		
		Define several notations
		$$\nu = \max\left\{ \frac{1}{3L_{xy}}\omega, \frac{1}{4L_{xy}}\omega^{-1}  \right\}.$$
		
		Using the definition of $\Psi^k$, we get
		$$\Psi^0 = \frac{1}{\eta_x}\|x^0-x^*\|+\frac{1}{\eta_y}\|x^0-x^*\| + \frac{2}{\sigma_x}B_f(x^0,x^*)+\frac{2}{\sigma_y}B_g(y^0,y^*).$$
		
		Where $\eta_x = \min\left\{ \frac{1}{4(\hat{\mu_x} + \hat{L_x}\sigma_x)}, \frac{\omega}{4L_{xy}} \right\}$, $\eta_y = \min\left\{ \frac{1}{4(\hat{\mu_y} + \hat{L_y}\sigma_y)}, \frac{1}{4L_{xy}\omega} \right\}.$
		
		Rewriting it in terms of $L_{lx}$, $L_{x}$, $L_{ly}$, $L_{y}$, $\mu_{lx}$, $\mu_{x}$, $\mu_{ly}$, $\mu_{y}.$
		
		$\eta_x = \min\left\{ \frac{1}{2\mu_{x} + 8L_{x}\sigma_x}, \frac{\omega}{4L_{xy}} \right\}$, $\eta_y = \min\left\{ \frac{1}{2\mu_{y} + 8L_{y}\sigma_y}, \frac{1}{4L_{xy}\omega} \right\}.$
		
		$$\nu \theta^k \Psi^0 \leq \frac{\epsilon}{3}.$$
		
		It would be sufficient to take $N = k = \frac{1}{1-\theta}\log\left(\frac{3\Psi^0\nu}{\epsilon}\right)$.
		
		Finally, let us estimate the right part
		$$\delta_x, \delta_y \leq \frac{(1-\theta)^2\epsilon}{24\nu}.$$
		
		Define $E$ as $E=\frac{1}{2n}\max\left\{\frac{L_{lx}^2}{L_{x}} + \frac{2L_{lx}^2}{\mu_{x}} + L_{lx} - \mu_{lx}, \frac{L_{ly}^2}{L_{y}} + \frac{2L_{ly}^2}{\mu_{y}} + L_{ly} - \mu_{lx} \right\}.$
		
		Using definition of $\delta_x$ and $\delta_y$, we get
		$$\delta' = \frac{(1-\theta)^2\epsilon}{24E\nu}.$$
		
		Define $F_x$ and $F_y$ as $F_x = \frac{\nu}{2n(1-\theta)}\left( \frac{1}{\hat{L_x}}+\frac{\omega}{L_{xy}} \right)$, $F_y= \frac{\nu}{2n(1-\theta)}\left( \frac{1}{\hat{L_y}}+\frac{1}{L_{xy}\omega} \right).$

		Using the definitions of $\Sigma^2$, $\sigma_{F, r}^2$, $\sigma_{G, r}^2$ we get, that it would be sufficient to take $r_{f, i} = \left\lceil\frac{6F_x\sigma_{f, i}^2}{\epsilon}\right\rceil$ and  $r_{g, i} = \left\lceil\frac{6F_y\sigma_{g, i}^2}{\epsilon}\right\rceil$.
		
		Finally
		
		$$N_{comm} = NT =\mathcal{O}\left( \frac{1}{1-\theta}\kappa\log\left(\frac{\Psi^0\nu}{\epsilon}\right)\log\left(\frac{D'}{\epsilon}\right)\right),$$
		
		\begin{align*}
			N_{comp}^i &= N(r_{i, f}+r_{i, g}) 
			\\&= 2N + \mathcal{O}\left( \frac{\max\{\omega, \omega^{-1}\}}{nL_{xy}(1-\theta)^2\epsilon}\left( \left( \frac{1}{L_{x}}+\frac{\omega}{L_{xy}} \right)\sigma_{f, i}^2 + \left( \frac{1}{L_{y}}+\frac{1}{L_{xy}\omega} \right)\sigma_{g, i}^2 \right)\log\left(\frac{\Psi^0\nu}{\epsilon}\right)\right).
		\end{align*}
		
		\begin{align*}\label{ests}
			\frac{1}{1-\theta} &=
			\mathcal{O}\left(\max\left\{\sqrt{\frac{L_{x}}{\mu_{x}}}, \sqrt{\frac{L_{y}}{\mu_{y}}},\frac{L_{xy}}{\sqrt{\mu_{x}\mu_{y}}}\right\}\right),
			\\
			\omega &= \sqrt{\frac{\mu_{y}}{\mu_{x}}}, \sigma_x = \sqrt{\frac{\mu_{x}}{8L_{x}}},\sigma_y = \sqrt{\frac{\mu_{y}}{8L_{y}}},\\
			\frac{1}{1-\theta} & =
			\mathcal{O}\left(\max\left\{
			\frac{\sqrt{L_{x}L_{y}}}{\mu_{xy}},
			\frac{L_{xy}}{\mu_{xy}}\sqrt{\frac{L_{x}}{\mu_{x}}},
			\frac{L_{xy}^2}{\mu_{xy}^2}
			\right\}\right), 
			\\
			\omega &= \sqrt{\frac{\mu_{xy}^2}{2\mu_{x}L_{x}}}, \sigma_x = \sqrt{\frac{\mu_{x}}{8L_{x}}},\sigma_y =\min\left\{1,\sqrt{\frac{\mu_{xy}^2}{16L_{x}L_{y}}}\right\}, \\
			\frac{1}{1-\theta} &=
			\mathcal{O}\left(\max\left\{
			\frac{\sqrt{L_{x}L_{y}}}{\mu_{yx}},
			\frac{L_{xy}}{\mu_{yx}}\sqrt{\frac{L_{y}}{\mu_{y}}},
			\frac{L_{xy}^2}{\mu_{yx}^2}
			\right\}\right),
			\\
			\omega &= \sqrt{\frac{2\mu_{y}L_{y}}{\mu_{yx}^2}}, \sigma_x =\min\left\{1,\sqrt{\frac{\mu_{yx}^2}{16L_{x}L_{y}}}\right\},\sigma_y = \sqrt{\frac{\mu_{y}}{8L_{y}}},\\
			\frac{1}{1-\theta} &= \mathcal{O}\left(\max\left\{
			\frac{\sqrt{L_{x}L_{y}}L_{xy}}{\mu_{xy}\mu_{yx}},
			\frac{L_{xy}^2}{\mu_{yx}^2},
			\frac{L_{xy}^2}{\mu_{xy}^2}
			\right\}\right),
			\\
			\omega &= \frac{\mu_{xy}}{\mu_{yx}}\sqrt{\frac{L_{y}}{L_{x}}}, \sigma_x = \min\left\{1,\sqrt{\frac{\mu_{yx}^2}{16L_{x}L_{y}}}\right\},\sigma_y =\min\left\{1,\sqrt{\frac{\mu_{xy}^2}{16L_{x}L_{y}}}\right\}.
		\end{align*}
	\end{document}

%% file: header.tex
\usepackage[english]{babel}
\usepackage{mathtools}
\usepackage{algorithm}
\usepackage{algpseudocode}
\usepackage{amsmath}
\usepackage{amssymb}
\usepackage{amsthm}
\usepackage[unicode,pdftex]{hyperref}
\usepackage{hyperref}
\usepackage[numbers]{natbib}
\usepackage{xcolor}

\newtheorem{theorem}{Theorem}[section]

\newtheorem{lemma}[theorem]{Lemma}
\newtheorem{assumption}[theorem]{Assumption}
\newtheorem{definition}[theorem]{Definition}
\newtheorem{remark}[theorem]{Remark}

\def\<#1,#2>{\langle #1,#2\rangle}

\newcommand{\bg}{B}

\DeclareMathOperator{\range}{range}

\newcommand{\R}{\mathbb{R}}

\newcommand{\E}{\mathbb{E}}
\newcommand{\D}{\mathbb{D}}

\newcommand{\mA}{{\bf A}}

\newcommand{\cC}{{\mathcal{C}}}

\renewcommand{\epsilon}{\varepsilon}

\newcommand{\norm}[1]{\left\| #1 \right\|}

\newcommand{\sqn}[1]{\left\|#1\right\|^2}
\newcommand{\angles}[1]{\left\langle #1 \right\rangle}

\newcommand{\braces}[1]{\left\{ #1 \right\}}

\newcommand{\ovl}[1]{\overline{#1}}